\documentclass[11pt]{article}
\usepackage[pagebackref,colorlinks=true,urlcolor=blue,linkcolor=red,citecolor=red]{hyperref}

\usepackage{epsfig, graphics, graphicx}
\usepackage{subcaption, comment, bm}
\usepackage[percent]{overpic}
\usepackage{float}
\usepackage{amssymb,bm}
\usepackage{amsthm}
\usepackage{color}
\usepackage[]{amsmath}
\usepackage[]{amsfonts}
\usepackage[]{fancyhdr}
\usepackage[]{graphicx,wrapfig}
\usepackage{enumitem}
\usepackage[utf8]{inputenc}
\usepackage[T1]{fontenc}
\usepackage{mathtools}
\usepackage{array}
\usepackage{comment}
\usepackage{amsmath}
\usepackage{tikz}
\usepackage{epigraph}
\usepackage{lipsum}
\allowdisplaybreaks

\definecolor{titlepagecolor}{cmyk}{1,.60,0,.40}
\patchcmd{\subsection}{\normalfont}{\normalfont\color{black}}{}{}
\DeclareFixedFont{\titlefont}{T1}{ppl}{b}{it}{0.5in}

\def\th@plain{%
  \thm@notefont{}% same as heading font
  \itshape % body font
}
\def\th@definition{%
  \thm@notefont{}% same as heading font
  \normalfont % body font
}
\makeatother

\usepackage{srcltx} % This is for PC only
\usepackage{amscd}
\usepackage{amssymb}
\usepackage{amsmath}
\usepackage{pb-diagram}
\usepackage{graphicx}
\usepackage{hyperref}
\usepackage{array}
\usepackage[all]{xy}
\xyoption{matrix}
\xyoption{arrow}
\usepackage{pdfsync}
\usepackage{physics}
\usepackage{enumitem}
 
\usepackage{thmtools}
\usepackage{amssymb}

\DeclareUnicodeCharacter{2212}{-,+}

\usepackage[none]{hyphenat}
\theoremstyle{definition}
\newtheorem{definition}{Definition}[section]
\newtheorem{theorem}[definition]{Theorem}

\newtheorem{prop}[definition]{Proposition}

\newtheorem{remark}{Remark}

\newcommand{\Lc}{\mathcal{L}}
\newcommand{\Mc}{\mathcal{M}}

\newcommand{\Rc}{\mathcal{R}}

\newcommand{\Tc}{\mathcal{T}}

\newcommand{\Vc}{\mathcal{V}}

\newcommand{\vx}{\textbf{\textit{x}}}

\newcommand{\vv}{\textbf{\textit{v}}}

\newcommand{\vf}{\textbf{\textit{f}}}

\newcommand{\I}{\mathrm{i}}
\newcommand{\D}{\mathrm{d}}

\newcommand{\vuu}{\textbf{\textit{u}}}

\newcommand{\Db}{\mathbb{D}}
\newcommand{\Rb}{\mathbb{R}}
\newcommand{\Zb}{\mathbb{Z}}

\title{\vspace{-1cm} Tensor tomography using V-line transforms with vertices restricted to a circle}

\author{Rohit Kumar Mishra\thanks{Department of Mathematics, Indian Institute of Technology, Gandhinagar, Gujarat, India. \url{rohit.m@iitgn.ac.in}, \url{rohittifr2011@gmail.com} }\and Anamika Purohit\thanks{Department of Mathematics, Indian Institute of Technology, Gandhinagar, Gujarat, India.
\url{anamika.purohit@iitgn.ac.in }}
\and Indrani Zamindar\thanks{Department of Mathematics, Indian Institute of Technology, Gandhinagar, Gujarat, India. \url{indranizamindar@iitgn.ac.in}}}

\begin{document}
\maketitle
\begin{abstract}
In this article, we study the problem of recovering symmetric $m$-tensor fields (including vector fields) supported in a unit disk $\Db$ from a set of generalized V-line transforms, namely longitudinal, transverse, and mixed V-line transforms, and their integral moments. We work in a circular geometric setup, where the V-lines have vertices on a circle, and the axis of symmetry is orthogonal to the circle. We present two approaches to recover a symmetric $m$-tensor field from the combination of longitudinal, transverse, and mixed V-line transforms. With the help of these inversion results, we are able to give an explicit kernel description for these transforms. We also derive inversion algorithms to reconstruct a symmetric $m$-tensor field from its first $(m+1)$ moment longitudinal/transverse V-line transforms.
%In the first method, we have the inversion formula using the decomposition introduced in \cite{derevtsov3}; in the second method, we have the componentwise reconstruction. 
\end{abstract}
%%%%%%%%%%%%%%%%%
\section{Introduction}
The V-line transform is a generalization of the Radon transform that maps a function to its integral along V-shaped trajectories. The study of the V-line transform, and its generalizations in different geometric settings is an active field of research due to their appearance in various imaging fields, viz, single scattering optical tomography \cite{FMS-PhysRev-10, FMS-11, FMS-09}, single scattering x-ray tomography \cite{Kats_Krylov-15}, single photon emission computed tomography \cite{basko1997analytical,basko1998application}, etc., where these transforms serve as a mathematical basis of the imaging models. Also, in many cases, researchers study such transformations out of mathematical interest due to intriguing connections with other areas of mathematics, viz, PDEs, microlocal analysis, differential geometry, etc. For a detailed discussion of the subject, please see the recent book \cite{amb-book}, where the author discusses several such generalized Radon transforms and their applications in various imaging modalities. 
 
 %In this article, we focus on the study of emission tomography, especially single photon emission computed tomography (SPECT), using a Compton camera. 
The study of emission tomography, especially single photon emission computed tomography (SPECT), involves Compton cameras and weakly radioactive tracers. The data obtained by a Compton camera is an integral of the tracer distribution along the cones with vertices on the scatter detector. For more details on imaging with Compton camera and its applications, please see  \cite{singh1983-partI, singh1983-partII, Fatma_Kuchment_Kunyansky_2018} and references therein.
To the best of our knowledge, the three-dimensional Conical Radon Transform (CRT) was first studied by Cree and Bones \cite{cree1994towards}. Later in  \cite{Moon_Haltmeir_CRT_2017}, the authors considered the CRT in three dimensions with vertices on a cylinder and the axis of symmetry pointing to the axis of the cylinder; in \cite{schiefeneder2017radon}, an $n$-dimensional CRT with vertices on a sphere and the axis of symmetry orthogonal to the sphere is studied. In \cite{Kuchment_Fatma_2017_divegentbeam_crt}, the authors took weighted cone transforms and derived inversion formulas by considering the detectors/sources lie on a submanifold which satisfies the Tuy's condition. Several authors studied the analytical reconstructions from the conical transform and gave their numerical validation with different geometric setups; please see \cite{Kuchment_homeland,gouia2014exact, haltmeier2014exact,jung2015inversion, jung2016exact, moon2016determination, nguyen2005radon,schonfelder1993_astro} and references therein.
 
The authors in \cite{basko1997analytical} proposed one-dimensional Compton camera imaging, where the conical surfaces reduce to  V-lines. Therefore, in this setup, the obtained integral data is an average of the gamma-ray distribution over the V-shaped lines, which is also known as the V-line transform of the unknown distribution function. The V-line transform of a scalar function in a circular geometry setup, where the vertex lies on a circle and the symmetry axis is orthogonal to the circle, has been investigated in \cite{Moon_Haltmeir_CRT_2017}. The attenuated V-line transform in the same setup has been considered in \cite{Moon_Haltmeir_Daniela}. The V-line transform with vertices on a line and fixed axis of symmetry has been studied in \cite{Truong-v-line}.
In \cite{Ambartsoumian_2012, Ambartsoumian_2013, Gaik_Souvik_numerics}, the V-line transform is considered in a circular geometry setup where the rays enter radially and, after traveling a distance, scatter at a fixed angle. A similar setup has been considered in a recent work \cite{bhardwaj2024inversion} for the recovery of vector fields using a set of generalized V-line transforms. The V-line transform and its generalizations for scalar/vector/tensor fields with fixed opening angles, fixed axis of symmetry, and variable vertex locations have been studied in a series of recent works \cite{amb-lat_2019, Amb_Lat_star, Gaik_Mohammad_Rohit, Gaik_Mohammad_Rohit_numerics, Gaik_Rohit_Indrani, Gaik_Rohit_Indrani_numerics}. In addition to these geometric settings, many authors investigated the broken ray/V-line transform when the domain contains an obstacle (known as a reflector). The ray enters the domain and gets reflected from the boundary of the obstacle, which generates a broken ray/V-line transform; please see \cite{Ilmavirta_tensor, Ilmavirta_function} and references therein.

As a natural extension of \cite{Moon_Haltmeir_CRT_2017}, we consider the question of recovering a symmetric $m$-tensor field from its generalized V-line transforms, where the vertex lies on a circle and axis of symmetry is orthogonal to the circle, that is, it passes through the center of the circle. To the best of our knowledge, no result is known for such generalization in this setup. In this work, we define longitudinal, transverse, and mixed V-line transforms and their integral moments for a symmetric $m$-tensor field. We derive various inversion formulas using different combinations of the defined transforms. Our analysis primarily relies on establishing relations between the defined generalized V-line transforms and the classical Radon transform, and then, using the known Fourier inversions of the Radon transform, we derive algorithms to reconstruct a symmetric $m$-tensor field explicitly.

The rest of the article is organized as follows. In Section \ref{sec: defs and notations}, we introduce the necessary notations and define the generalized V-line transforms of our interest. Section \ref{sec: Main_result} presents the main results of this article along with a brief discussion about them. In Section \ref{sec: decomp recovery} and \ref{sec: Compwise recovery}, we present two different approaches for the recovery of a symmetric $m$-tensor field from the combinations of longitudinal, transverse, and mixed V-line transforms. Additionally, using derived inversion results, we give an explicit kernel description for these transforms. Section \ref{sec: moments recovery} is devoted to recovering a symmetric $m$-tensor field from its first $(m+1)$ moment longitudinal/transverse V-line transforms. We conclude the article with acknowledgments in Section \ref{sec: acknowledge}.

\section{Definitions and notations}\label{sec: defs and notations}
The unit disk centered at the origin in $\Rb^2$ is denoted by $\Db$.
For $m\geq 1$, let $S^m(\Db)$ denote the space of symmetric $m$-tensor fields defined in $\Db$ (here $m=1$ corresponds to vector fields), and $C_c^{\infty}(S^m(\Db))$ be the collection of infinitely differentiable, compactly supported, symmetric $m$-tensor fields. Throughout the paper, we use bold font letters to denote vector and tensor fields in $\Rb^2$ and regular font letters to denote scalar functions. In local coordinates, $\vf\in C_c^{\infty}(S^m(\Db))$ can be expressed as 
$$\vf(\vx) = f_{i_1\dots i_m}(\vx)dx^{i_1} \dots  dx^{i_m},$$
where $f_{i_1\dots i_m}(\vx)$ are compactly supported smooth functions which are symmetric in their indices.  Note that the sum over all the repeated indices $i_1, \dots, i_m$   is assumed (Einstein summation convention) here and in the upcoming text. 
%The scalar product in $S^m(B)$ is defined by 
%$$\langle \vf, \textbf{\textit{g}}\rangle =f_{i_1i_2\dots i_m}g_{i_1i_2\dots i_m}.$$
The scalar product in $S^m(\Db)$ is defined by the formula
$$
\langle f(\vx), g(\vx)\rangle=f_{i_1 \ldots i_m}(\vx) g_{i_1 \ldots i_m}(\vx).
$$
The classical gradient ($\D$), and its orthogonal operator ($\D^\perp$) for a scalar function $V(x_1, x_2)$ and divergence ($\delta$) and corresponding orthogonal operator ($\delta^\perp$) for a vector field $\vf =(f_1,f_2)$ are defined as follows: 
\begin{align*}%\label{eq: definition of div and curl}
\D V = \left(\frac{\partial V}{\partial x_1}, \frac{\partial V}{\partial x_2}\right), \ \  \D^\perp V = \left(-\frac{\partial V}{\partial x_2}, \frac{\partial V}{\partial x_1}\right), \ \    \delta \vf =  \frac{\partial f_1}{\partial x_1}+ \frac{\partial f_2}{\partial x_2},\ \  
\delta^\perp \vf =  \frac{\partial f_2}{\partial x_1}- \frac{\partial f_1}{\partial x_2}.
\end{align*}
\noindent There is a natural generalization of these differential operators to higher-order symmetric $m$-tensor fields. More precisely, we have maps $\D, \D^\perp: C_{c}^{\infty}(S^{m}(\Db))\rightarrow C_c^{\infty}(S^{m+1}(\Db))$ and $\delta, \delta^\perp: C_{c}^{\infty}(S^{m}(\Db))\rightarrow C_c^{\infty}(S^{m-1}(\Db))$ that are defined in the following way (see \cite{derevtsov3} for a detailed discussion on this):
% The generalization of the operators, gradient ($\D$), orthogonal gradient ($\D^\perp$), divergence ($\delta$), and orthogonal divergence ($\delta^\perp$) to higher-order tensor fields, where  $\D, \D^\perp: C_{c}^{\infty}(S^{m}(\Db))\rightarrow C_c^{\infty}(S^{m+1}(\Db))$ and $\delta, \delta^\perp: C_{c}^{\infty}(S^{m}(\Db))\rightarrow C_c^{\infty}(S^{m-1}(\Db))$ can be defined in the following way\footnote{These generalized differential operators have been defined in \cite{derevtsov3}.}:
\begin{align*}%\label{eq:generalized differential operators}
(\D\vf)_{i_1\dots i_{m+1}} &= \frac{1}{m+1}\left( \frac{\partial f_{i_1\dots i_{m}}}{\partial x_{i_{m+1}}}+ \sum_{k=1}^{m} \frac{\partial f_{i_1\dots i_{k-1}i_{m+1}i_{k+1}\dots i_{m}}}{\partial x_{i_k}}\right)\\
(\D^\perp \vf)_{i_1\dots i_{m+1}} &= \frac{1}{m+1}\left((-1)^{i_{m+1}} \frac{\partial f_{i_1\dots i_{m}}}{\partial x_{3-{i_{m+1}}}}+ \sum_{k=1}^{m} (-1)^{i_k} \frac{\partial f_{i_1\dots i_{k-1}i_{m+1}i_{k+1}\dots i_{m}}}{\partial x_{3-{i_k}}}\right)\\
(\delta \vf)_{i_1\dots i_{m-1}} &= \frac{\partial f_{i_1\dots i_{m-1}1}}{\partial x_1} + \frac{\partial f_{i_1\dots i_{m-1}2}}{\partial x_2} =  \frac{\partial f_{i_1\dots i_{m-1}i_m}}{\partial x_{i_m}}\\
(\delta^\perp \vf)_{i_1\dots i_{m-1}} &= -\frac{\partial f_{i_1\dots i_{m-1}1}}{\partial x_2} + \frac{\partial f_{i_1\dots i_{m-1}2}}{\partial x_1} = (-1)^{i_m} \frac{\partial f_{i_1\dots i_{m-1}i_m}}{\partial x_{3-{i_m}}}.
\end{align*}
%%%%%%%%%%%%%% V-line setup%%%%
For $\phi \in [0,2\pi)$, let $ \Phi(\phi):= (\cos\phi, \sin\phi)$ denotes the unit vector corresponding to the polar angle $\phi$. Further, for $\psi \in (0,\pi/2)$, we define two linearly independent unit vectors $\vuu (\phi, \psi) = \Phi(\pi + \phi - \psi)$ %$  (-\cos(\phi-\psi), - \sin(\phi-\psi))$
and $ \vv(\phi, \psi)  = \Phi(\pi + \phi + \psi)$. %$(-\cos(\phi+\psi), - \sin(\phi+\psi))$. 
Given $(\phi, \psi) \in [0, 2\pi) \times (0, \pi/2)$, we consider a V-line starting at $\Phi(\phi)$ with one branch along $\vuu (\phi, \psi)$ and the other branch along $\vv (\phi, \psi)$, see Figure \ref{fig1(a)}. Therefore, $(\phi, \psi)$ parametrize the space of $V$-lines (considered in this article) as follows: a V-line for us is a union of two rays starting at $\Phi(\phi)$ in the directions $\vuu(\phi, \psi)$ and $\vv(\phi, \psi)$, respectively, that is, a V-line can be written as (see Figure \ref{fig1(a)})
$$ V(\phi, \psi):=\{\Phi(\phi)+t\vuu(\phi, \psi): 0\le t <\infty \}\cup \{\Phi(\phi)+t\vv(\phi, \psi): 0\le t <\infty \}.$$
\begin{figure}[H] 
\centering
\includegraphics[width=1.0\textwidth]{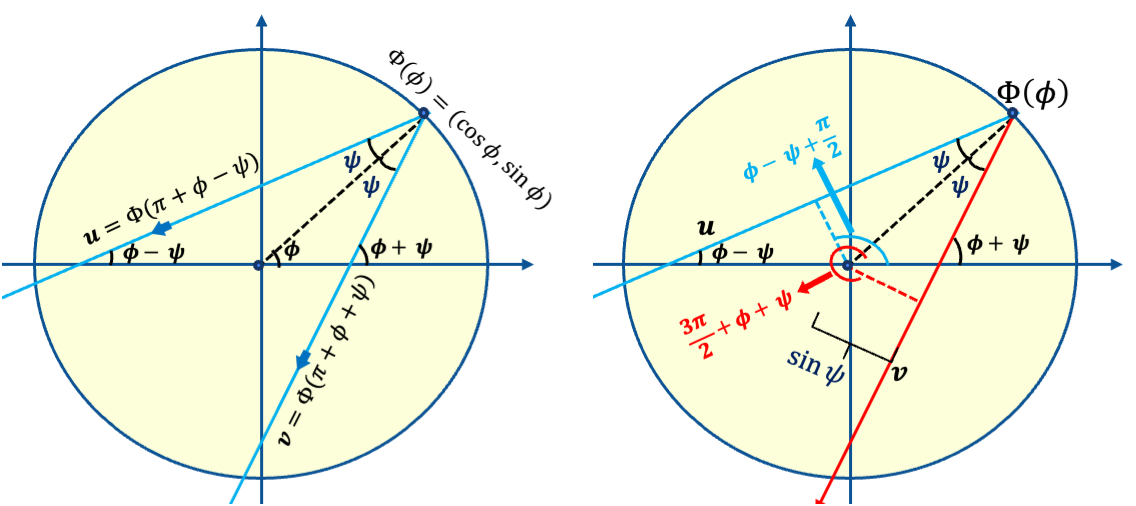}
\caption{A V-line with vertex $\Phi(\phi)=(\cos{\phi},\sin{\phi})$ and directions $\vuu =  \Phi(\pi + \phi - \psi)$ and $\vv =  \Phi(\pi + \phi + \psi)$. Both the rays are at a distance $s = \sin \psi$ from the origin and are orthogonal to unit vectors $ \Phi(\phi - \psi+ \pi/2)$ and $ \Phi(\phi + \psi+ 3\pi/2) = \Phi(\phi + \psi -\pi/2)$, respectively.} \label{fig1(a)}
\end{figure}
%  \begin{figure}[H]
%      \centering
%      \begin{subfigure}[b]{0.44\textwidth}
%          \centering
%          \includegraphics[width=1.1\textwidth]{Pictures/V-line on circle.png}
%          \caption{V-line with vertex $\Phi(\phi)=(\cos{\phi},\sin{\phi})$ on the circle, symmetry axis $\{(1-t)\Phi(\phi)| t>0\}$ normal to the circle and half opening angle $\psi$.}
%   \label{fig1(a)}
%      \end{subfigure}
%      \hfill
%      \begin{subfigure}[b]{0.46\textwidth}
%          \centering
%          \includegraphics[width=1.15\textwidth]{Pictures/V-line circular setup.png}
%          \caption{Data corresponding to the V-line is equal to the sum of the integrals along two straight lines with angle $\xi^{-}(\phi,\psi)$ and $\xi^{+}(\phi,\psi).$}
%   \label{fig1(b)}
%      \end{subfigure}
%      \caption{}
% \end{figure}
\vspace{-0.5cm}
% \begin{remark}
% \begin{itemize}
% \item[]
% \item It is unnecessary to keep the arguments $\phi$ and $\psi$ in the upcoming long expressions involving the components of vectors $\vuu (\phi,\psi) = (u_1(\phi,\psi) , u_2(\phi,\psi) )$ and $\vv(\phi,\psi)  = (v_1(\phi,\psi) , v_2(\phi,\psi) )$. Therefore, for readability and simplifying expressions, we choose to write $\vuu$ and $\vv$ in the place of $\vuu (\phi, \psi)$ and $\vv (\phi, \psi)$, respectively, with the understanding that $\vuu$ and $\vv$ are vector fields depending on $\phi$ and $\psi$.
% \item 
% \end{itemize}
% \end{remark}
\noindent   Below, we present a very brief discussion on the Radon transform in $\Rb^2$ and two inversion methods involving Fourier series expansions. 
\begin{definition}\label{regular Radon}\cite{natterer2001mathematics}
For $h \in C_c^\infty(\Db)$, we define the \textbf{Radon transform} of $h$ in the following way:
\begin{align}
\Rc h(s,\phi) = \Rc h\left(s,\Phi(\phi)\right)= \int_{-\infty}^{\infty}h\left(s\Phi(\phi)+t\Phi(\phi)^\perp\right)\,dt, \quad \phi \in [0, 2\pi),\quad  s \in \Rb.
\end{align}
\end{definition}

% \noindent We give a quick review of the inversion of the Radon transform using Fourier series expansion, which will be used at a later stage. 
% \noindent Next, we recall the two-dimensional (regular) Radon transform of a function $h$, which we will frequently use to get the inversion formulas. 
\noindent Let us expand the unknown function $h$, and Radon transform $\Rc h $ using the Fourier series with respect to the angular variables as follows: 
 \begin{align}\label{fourier component function}
h(r\Phi(\phi))&=\sum_{n\in \Zb} h_n(r)e^{\I n\phi}, \qquad \ \  \text{with} \qquad \ \ \  h_n(r)= \frac{1}{2\pi}\int_0^{2\pi}h(r\Phi(\phi)) e^{-\I n\phi}\,d\phi.\\\label{fourier component known data}
\Rc h(s,\phi)&=\sum_{n\in \Zb} (\Rc h)_n(s)e^{\I n\phi}, \quad \text{with} \quad (\Rc h)_n(s)= \frac{1}{2\pi}\int_0^{2\pi}(\Rc h)(s,\phi) e^{-\I n\phi}\,d\phi.
 \end{align}
Then, there are known explicit inversion formulas for the Radon transform to recover $h_n$ in terms of $(\Rc h)_n$ derived by A.M. Cormack \cite{cormack1963representation} and R.M. Perry \cite{perry1975reconstructing}, respectively:
\begin{align}\label{Recovery formula-I}
h_n(r) = -\frac{1}{\pi}\int_r^1 \frac{d(\Rc h)_n(s)}{d s} \frac{T_{|n|}(s/r)}{\sqrt{s^2-r^2}}\, ds
\end{align}
and
\begin{align}\label{Recovery formula-II}
h_n(r) =- \frac{1}{\pi r}\bigg[\int_r^1 \frac{d (\Rc h)_n(s)}{d s} \frac{\left[s/r+ \sqrt{s^2/r^2-1}\right]^{-|n|}}{\sqrt{s^2/r^2- 1}}\, ds - \int_0^r\frac{d (\Rc h)_n(s)}{d s} U_{|n|-1}(s/r)\, ds\bigg],
\end{align}
where for $k\ge 0$, $T_k(x), U_k(x)$ denote Chebyshev polynomials of the first and second kind, respectively, and are defined as
\begin{align*}
 T_k(x) &= \begin{cases}
   \cos(k \cos^{-1}(x)), & |x|\le1 \\
 \cosh(k\cosh^{-1}(x)),  &  \ \ x > 1\\
 (-1)^k\cosh(k \cosh^{-1}(-x)), & \ \  x < -1
	\end{cases}\\
 U_k(x) &= \begin{cases}
   \sin((k+1) \cos^{-1}(x))/\sin(\cos^{-1}(x)), & |x|\le1 \\
 \sinh((k+1)\cosh^{-1}(x))/\sinh(\cosh^{-1}(x)), & \ \  x > 1\\
(-1)^k \sinh((k+1)\cosh^{-1}(-x))/\sinh(\cosh^{-1}(-x)) , &\ \  x<-1
	\end{cases}
\end{align*}
with $U_{-1}=0.$

In addition to these inversion formulas, we will also use the following property of Radon transform repeatedly throughout the article:
\begin{align}
    \Rc\left(\frac{\partial h}{\partial x_k}\right)(s,\Phi(\phi))= \left(\Phi(\phi)\right)_{k}\frac{\partial}{\partial s}\Rc h\left(s,\Phi(\phi)\right), \quad k=1,2.
\end{align}
Next, we introduce a weighted V-line transform of a scalar function $h$ and its inversion formula using the inversion of the Radon transform discussed above. 
\begin{definition}\label{weighted V-line}
For $h \in C_c^\infty(\Db)$, a \textbf{weighted V-line transform} of $h$ is defined by
\begin{align}
\Vc_{w}h(\phi,\psi)= c_1\int_{0}^{\infty} h(\Phi(\phi)+t \vuu)\,dt +c_2 \int_{0}^{\infty} h(\Phi(\phi)+t \vv)\,dt,
\end{align}
where $c_1$ and $c_2$ are non-zero constants. 
\end{definition}
% \noindent Since $h$ is compactly supported inside the disk $\Db$ and vertices of V-lines are only on the boundary of $\Db$, we have the following relation 
% \begin{align}\label{eq:relation between Radon and W V line}
%      \Vc_{w}h(\phi,\psi) 
%      %&= c_1\Rc h\left(\sin
%      % {\psi},\phi-\psi+\frac{\pi}{2}\right)+c_2 \Rc h\left(\sin{\psi},\phi+\psi+\frac{3\pi}{2}\right)\nonumber\\
%      = c_1\Rc h\left(\sin{\psi},\phi-\psi+\frac{\pi}{2}\right)+c_2 \Rc h\left(\sin{\psi},\phi+\psi-\frac{\pi}{2}\right).
% \end{align}
The following result gives two inversion formulas to recover a scalar function from its weighted V-line transform. These formulas are derived directly using the inversion formulas \eqref{Recovery formula-I} and \eqref{Recovery formula-II}. Still, we present a quick proof for the sake of completeness and also because we will use them later.
\begin{theorem}
    Let $h\in C^{\infty}_c(\Db)$, then $h$ can be recovered from $\Vc_{w}h$ from any of the following inversion formulas:
    \begin{align}\label{weighted Recovery formula-I}
h_n(r) = -\frac{1}{\pi}\int_r^1 \frac{d}{d s}\left[
 \frac{(\Vc_{w}h)_{n}(\sin^{-1}{s})}{c_1 e^{-\I n(\sin^{-1}{s}- \pi/2)} +c_2  e^{\I n(\sin^{-1}{s}- \pi/2)}}\right]\frac{T_{|n|}(s/r)}{\sqrt{s^2-r^2}}\, ds,
\end{align}
\begin{align}\label{weighted Recovery formula-II}
h_n(r) =- \frac{1}{\pi r}\bigg[\int_r^1 \frac{d}{d s}\left[
 \frac{(\Vc_{w}h)_{n}(\sin^{-1}{s})}{c_1 e^{-\I n(\sin^{-1}{s}- \pi/2)} +c_2  e^{\I n(\sin^{-1}{s}- \pi/2)}}\right] \frac{\left[s/r+ \sqrt{s^2/r^2-1}\right]^{-|n|}}{\sqrt{s^2/r^2- 1}}\, ds \nonumber\\ - \int_0^r\frac{d}{d s}\left[
 \frac{(\Vc_{w}h)_{n}(\sin^{-1}{s})}{c_1 e^{-\I n(\sin^{-1}{s}- \pi/2)} +c_2  e^{\I n(\sin^{-1}{s}- \pi/2)}}\right] U_{|n|-1}(s/r)\, ds\bigg],
\end{align}
 where  $h_n$ and $(\Vc_{w}h)_{n}$ are the $n^{th}$ Fourier coefficients of the $h$ and $\Vc_{w}h$, respectively. 
\end{theorem}
\begin{proof}
Since $h$ is compactly supported inside the disk $\Db$ and vertices of V-lines are only on the boundary of $\Db$, we can express weighted V-line transform in terms of the Radon transform as follows: 
\begin{align}\label{eq:relation between Radon and W V line}
     \Vc_{w}h(\phi,\psi) 
     %&= c_1\Rc h\left(\sin
     % {\psi},\phi-\psi+\frac{\pi}{2}\right)+c_2 \Rc h\left(\sin{\psi},\phi+\psi+\frac{3\pi}{2}\right)\nonumber\\
     = c_1\Rc h\left(\sin{\psi},\phi-\psi+\frac{\pi}{2}\right)+c_2 \Rc h\left(\sin{\psi},\phi+\psi-\frac{\pi}{2}\right).
\end{align}
% \noindent \textbf{Inversion of weighted V-line transform}\\
Using the definition of Fourier coefficients (defined above in equations \eqref{fourier component function} and \eqref{fourier component known data}), we have
\begin{align*}
   (\Vc_{w}h)_{n}(\psi)&=c_1 (\Rc h)_{n}(\sin{\psi})e^{-\I n(\psi- \pi/2)} +c_2 (\Rc h)_{n}(\sin{\psi}) e^{\I n(\psi- \pi/2)} \\
   &= \left(c_1 e^{-\I n(\psi- \pi/2)} +c_2  e^{\I n(\psi- \pi/2)}\right) (\Rc h)_{n}(\sin{\psi})\\
   \Longrightarrow\qquad \qquad \qquad  (\Rc h)_{n}(\sin{\psi}) &= \frac{(\Vc_{w}h)_{n}(\psi)}{c_1 e^{-\I n(\psi- \pi/2)} +c_2  e^{\I n(\psi- \pi/2)}}.
\end{align*}
Then, we obtain the required inversion formulas \eqref{weighted Recovery formula-I} and \eqref{weighted Recovery formula-II}  by a direct application of \eqref{Recovery formula-I} and \eqref{Recovery formula-II}.
\end{proof}
\begin{remark}
\begin{itemize}
\item[]
% \item Since we are dealing with angles and $\Rc h\left(\sin{\psi},\phi+\psi+\frac{3\pi}{2}\right)=\Rc h\left(\sin{\psi},\phi+\psi-\frac{\pi}{2}\right),$ so from now onwards we will work with the angle $\phi+\psi-\frac{\pi}{2}.$
\item For the special choice of $c_1=c_2=1$, the transform $\Vc_w$ reduces to the standard V-line transform $\Vc$ and for $c_1 = -c_2 = 1$, $\Vc_w$ reduces to signed V-line transform $\Vc^{-}$. 
\item For a scalar function $h$, defined on $\Db$, its Radon transform is assumed to be zero for lines which are at a distance strictly bigger than $1$ from the origin, that is, $\Rc h(s,\phi) = 0$ for $|s| > 1$. 
\item It is unnecessary to keep the arguments $\phi$ and $\psi$ in the upcoming long expressions involving the components of vectors $\vuu (\phi,\psi) = (u_1(\phi,\psi) , u_2(\phi,\psi) )$ and $\vv(\phi,\psi)  = (v_1(\phi,\psi) , v_2(\phi,\psi) )$. Therefore, for better readability and simplifying expressions, we choose to write $\vuu$ and $\vv$ in the place of $\vuu (\phi, \psi)$ and $\vv (\phi, \psi)$, respectively, with the understanding that $\vuu$ and $\vv$ are vector fields depending on $\phi$ and $\psi$. For the same reason to reduce the length of many expressions, we define $\xi^{+}(\phi,\psi)=\phi+\left(\psi-\frac{\pi}{2}\right)$ and $\xi^{-}(\phi,\psi)=\phi-\left(\psi-\frac{\pi}{2}\right)$. 
\end{itemize}  
\end{remark}
\noindent Our aim in this article is to develop inversion algorithms to recover symmetric $m$-tensor fields from various generalized V-line transforms, which we introduce below. 
%%%%%% V-line transform definitions%%%%%%
\begin{definition}\label{longi V-line}
Let $\vf \in C_c^{\infty}(S^m(\Db)).$ The \textbf{longitudinal V-line transform} of $\vf$ is defined by
\begin{align}\label{eq:longi V-line}
\Lc\vf(\phi,\psi)&=\int_{0}^{\infty} \langle \vuu^m, \vf(\Phi(\phi)+t \vuu) \rangle \,dt +\int_{0}^{\infty} \langle \vv^m, \vf(\Phi(\phi)+t \vv) \rangle \,dt \nonumber\\
% &=\int_{-\infty}^{\infty} \langle\vuu^m, \vf\left(\sin (\psi)\Phi\left(\xi^{-}(\phi,\psi)\right) + t\vuu\right)  \rangle\,dt + \int_{-\infty}^{\infty} \langle \vv^m,\vf\left(\sin(\psi)\Phi\left(\xi^{+}(\phi,\psi)\right) + t\vv\right)\rangle\,dt
&=\Rc\left(\left\langle \vuu^m,  \vf  \right\rangle\right)\left(\sin {\psi},\xi^{-}(\phi,\psi) \right)+ \Rc\left(\left\langle \vv^m, \vf \right\rangle\right)\left(\sin {\psi},\xi^{+}(\phi,\psi) \right), 
\end{align}
where $\vuu^m$ denotes the $m^{th}$ symmetric tensor product of $\vuu$ and $ \langle \vuu^m, \vf \rangle = u_{i_1}\dots u_{i_m} f_{i_1\dots i_m}$.
\end{definition}
\begin{definition}\label{trans V-line}
Let $\vf \in C_c^{\infty}(S^m(\Db)).$ The \textbf{transverse V-line transform} of $\vf$ is defined by
\begin{align}\label{ed:trans V-line}
\Tc\vf(\phi,\psi)&=\int_{0}^{\infty} \left\langle (\vuu^\perp)^m, \vf(\Phi(\phi)+t \vuu) \right\rangle \,dt +\int_{0}^{\infty} \left\langle (\vv^\perp)^m, \vf(\Phi(\phi)+t \vv) \right\rangle \,dt \nonumber\\
% &=\int_{-\infty}^{\infty} \langle(\vuu^\perp)^m, \vf\left(\sin (\psi)\Phi\left(\xi^{-}(\phi,\psi)\right) + t\vuu\right)  \rangle\,dt + \int_{-\infty}^{\infty} \langle (\vv^\perp)^m,\vf\left(\sin(\psi)\Phi\left(\xi^{+}(\phi,\psi)\right) + t\vv\right)\rangle\,dt
&=\Rc\left(\left\langle (\vuu^\perp)^m,  \vf  \right\rangle\right)\left(\sin {\psi},\xi^{-}(\phi,\psi) \right)+ \Rc \left(\left\langle (\vv^\perp)^m,  \vf \right\rangle\right)\left(\sin {\psi},\xi^{+}(\phi,\psi) \right).
\end{align}
% \noindent where $(\vuu^\perp)^m$ denotes the $m^{th}$ symmetric tensor product of $\vuu^\perp$ and $ \langle (\vuu^\perp)^m, \vf \rangle = u^\perp_{i_1}\dots u^\perp_{i_m} f_{i_1\dots i_m}$ and 
Here $\vuu^\perp$ is defined by taking $90^{\circ}$ anticlockwise rotation of $\vuu=(u_1,u_2)$, i.e. $\vuu^\perp=(-u_2,u_1)$.
\end{definition}
\begin{definition}\label{mixed V-line}
    Let $\vf \in C_c^{\infty}(S^m(\Db))$ and $1\le \ell\le(m-1).$ The \textbf{mixed V-line transforms} of $\vf$ is defined by
\begin{align}\label{eq:mixed V-line}
\Mc_{\ell}\vf(\phi,\psi)&=  \int_{0}^{\infty}   \left\langle (\vuu^\perp)^\ell\vuu^{(m-\ell)}, \vf(\Phi(\phi)+t \vuu) \right\rangle \,dt +\int_{0}^{\infty} \left\langle (\vv^\perp)^\ell\vv^{(m-\ell)}, \vf(\Phi(\phi)+t \vv) \right\rangle \,dt\nonumber\\
% &=\int_{-\infty}^{\infty} \langle(\vuu^\perp)^\ell\vuu^{(m-\ell)}, \vf\left(\sin (\psi)\Phi\left(\xi^{-}(\phi,\psi)\right) + t\vuu\right)  \rangle\,dt + \int_{-\infty}^{\infty} \langle (\vv^\perp)^\ell\vv^{(m-\ell)},\vf\left(\sin(\psi)\Phi\left(\xi^{+}(\phi,\psi)\right) + t\vv\right)\rangle\,dt
&=\Rc\left(\left\langle (\vuu^\perp)^\ell\vuu^{(m-\ell)},  \vf \right\rangle\right)\left(\sin {\psi},\xi^{-}(\phi,\psi) \right) + \Rc\left(\left\langle (\vv^\perp)^\ell\vv^{(m-\ell)},  \vf \right\rangle\right)\left(\sin {\psi},\xi^{+}(\phi,\psi) \right),
\end{align}
where $\displaystyle (\vuu^\perp)^\ell\vuu^{(m-\ell)} = u^\perp_{i_1}\dots u^\perp_{i_\ell}u_{i_{\ell+1}}\dots u_{i_m}$.
\end{definition}
\begin{definition}\label{moment longi V-line}
Let $\vf \in C_c^{\infty}(S^m(\Db))$ and $k\ge 0$ be an integer. The \textbf{$k$-th moment longitudinal V-line transform} of $\vf$ is defined by
\begin{align}\label{eq:moment longi V-line}
\Lc^{k}\vf(\phi,\psi)=\int_{0}^{\infty} t^k \left\langle \vuu^m, \vf(\Phi(\phi)+t \vuu) \right\rangle \,dt +\int_{0}^{\infty} t^k \left\langle \vv^m, \vf(\Phi(\phi)+t \vv) \right\rangle \,dt. \end{align}
\end{definition}
\begin{definition}\label{moment trans V-line}
Let $\vf \in C_c^{\infty}(S^m(\Db))$ and $k\ge 0$ be an integer. The \textbf{$k$-th moment transverse V-line transform} of $\vf$ is defined by
\begin{align}\label{eq:moment trans V-line}
\Tc^{k}\vf(\phi,\psi)=\int_{0}^{\infty} t^k\left\langle (\vuu^\perp)^m, \vf(\Phi(\phi)+t \vuu) \right\rangle \,dt +\int_{0}^{\infty} t^k \left\langle (\vv^\perp)^m, \vf(\Phi(\phi)+t \vv) \right\rangle \,dt.     
\end{align}
\end{definition}
\noindent Now, we are ready to state the main results we addressed in the article. 
%%%%%%%%%%%%%%%%%%%%%%%%%%%%%%%%%%%%%%%%%%%%%%%%%%%%%%%%%%%

\begin{section}{Main results} \label{sec: Main_result}
\begin{theorem}\label{Recovery from M_lf}
Let $\vf\in C^{\infty}_c(S^m(\Db))$. Then $\vf$ can be recovered uniquely from the knowledge of $\Mc_{\ell}\vf$,   $0\leq \ell \leq m$, where $\Mc_{0}=\Lc$ and $\Mc_{m}=\Tc$. 
\end{theorem}
\noindent We present two approaches to prove this theorem, which we briefly discuss below.
\vspace{2mm}

\noindent \textbf{$1^{st}$ approach to prove Theorem \ref{Recovery from M_lf}:} The first method is based on a known decomposition of a symmetric $m$-tensor field in $\Rb^2$ (derived in \cite{derevtsov3}). This decomposition is a generalization of the well-known potential (curl-free) and solenoidal (divergence-free) decomposition of a vector field in $\Rb^2$. The main idea here is to derive appropriate relations between the generalized V-line transforms ($\Mc_{\ell}\vf$) and the weighted V-line transforms of the corresponding potentials with appropriate weights. Then, using any of the inversion formulas of a weighted V-line transform of a function given in \eqref{weighted Recovery formula-I} and \eqref{weighted Recovery formula-II}, we can recover the potentials explicitly.
% The main idea here is to derive appropriate relations between the generalized V-line transforms ($\Mc_{\ell}\vf$) and the Radon transforms of the corresponding potentials. Then, using the Fourier series expansion method, we calculated the Fourier coefficients of the Radon transforms of each potential, and from that, we recovered unknown potentials by using inversion of the Radon transform.
As a corollary of this inversion and the used approach, we also give an explicit kernel description for these generalized $V$-line transforms (discussed in the proof section).

% \begin{theorem}\label{th:kernel description}
%     Let $\vf \in C^{\infty}_c(S^m(\Db))$. Then $\Mc_{\ell}\vf=0$ if and only if $\displaystyle\vf = \sum_{j \neq \ell}(\D^{\perp})^{m-j}\D^{j}\chi^{(j)}$, $\forall \ 0\leq j,\ell\leq m$ for scalar functions $\chi^{(j)}$ satisfying the boundary conditions $ \chi^{(j)}|_{\partial \Db}=0,\dots,\D^{(m-1)}\chi^{(j)}|_{\partial \Db}=0$.
% \end{theorem}

%  We have proved this theorem in two different ways. In Section \ref{sec: Compwise recovery}, we show the componentwise reconstruction. There, we have written the known data ($\Mc_{\ell}\vf$) in terms of the Radon transform of each component of a symmetric $m$-tensor field. Then, we calculated the Fourier coefficients of the corresponding Radon transforms of the components from the Fourier series expansion method, and by using \eqref{Recovery formula-I} or \eqref{Recovery formula-II}, we recovered the components explicitly. In Section \ref{sec: decomp recovery}, we used the decomposition defined in \cite{derevtsov3} to recover a symmetric $m$-tensor field. There, we showed the relation between the defined generalized V-line transforms 
% ($\Mc_{\ell}\vf$) and the Radon transforms of the corresponding potentials. Then, by using the Fourier series expansion method, we calculated the Fourier coefficients of the Radon transforms of each potential, and from that, we recovered the respective potentials used in the decomposition of a symmetric $m$-tensor field.
\vspace{2mm}
\noindent \textbf{$2^{nd}$ approach to prove Theorem \ref{Recovery from M_lf}:} In the second strategy, we recover a symmetric $2$-tensor field componentwise. This approach uses the Fourier series expansion method similar to that discussed above for the Radon transform and weighted V-line transform. More specifically, we will expand the components of the unknown symmetric $2$-tensor field $\vf$, and its integral transforms in the Fourier series. Then, the goal will be to recover the Fourier coefficients of components of $\vf$ in terms of the Fourier coefficients of used integral transforms.
\begin{theorem}\label{Recovery from longi/trans moments}
Let $\vf\in C^{\infty}_c(S^m(\Db))$. Then, $\vf$ can be uniquely reconstructed from its first $(m+1)$ integral moment longitudinal/transverse V-line transforms.
\end{theorem}  
\noindent To address this question, we again use the decomposition used in the $1^{st}$ approach and derive relations between the defined first ($m+1$) moment longitudinal/transverse V-line transforms and weighted V-line transforms of the corresponding potentials with appropriate weights. Then, we use the known inversion formulas of a weighted V-line transform to conclude the argument.

%and the Radon transform of the corresponding potentials. Then, using the Fourier series expansion method and formula \eqref{Recovery formula-I}/\eqref{Recovery formula-II}, we can recover the potentials explicitly.
\end{section}

%%%%%%%%%%%%%%%%%%%%%%%%%%%%%%%%%%%%%%%%%%%%%%%%%%%%%%%%%%

\section{Proof of Theorem \ref{Recovery from M_lf} (Approach 1)}\label{sec: decomp recovery}
%In this section, we present the recovery of %vector fields,  symmetric $2$-tensor fields, and a symmetric $m$-tensor field from the knowledge of longitudinal, transverse, and mixed V-line transforms by using the decomposition formulas defined in \cite{derevtsov3}.
We break this section into two subsections addressing the cases of vector fields and symmetric $m$-tensor fields separately. The vector field can be treated as a practice case, which gives an idea of what to expect for the case of symmetric tensor fields 
of arbitrary order $m$. In addition to the inversion algorithms, we also give a characterization of the kernel of longitudinal, transverse, and mixed V-line transforms. 
% We discuss the proof of Theorem \ref{Recovery from M_lf}  for vector fields,  symmetric $2$-tensor fields, and symmetric $m$-tensor fields from the knowledge of longitudinal, transverse, and mixed V-line transform. We discuss the kernel of the respective generalized V-line transforms for vector and tensor fields and show relationships between the generalized V-line transforms and the Radon transform of the corresponding potentials in the propositions. Then, using the propositions and Fourier series expansion, we recovered the respective potentials used in the decomposition of a symmetric $m$-tensor field.
\subsection{Vector fields ($m=1$)}
We start by recalling a decomposition result for vector fields presented in \cite{derevtsov3}. 
\begin{theorem}\label{vector decomp}(\cite{derevtsov3})
For any $\vf \in C^{\infty}_c(S^1(\Db))$, there exist unique smooth functions (known as potentials) $\chi$ and $\eta$ such that 
\begin{align}
    \vf= \D^{\perp}\chi+\D\eta, \quad \chi|_{\partial \Db}=0, \quad \eta|_{\partial \Db}=0. 
\end{align}
The following proposition is a generalization of the \cite[Proposition 3.1]{derevtsov3} in the V-line setting. The proof follows along exactly similar lines, and therefore, we did not present it here. 
\end{theorem}
\begin{prop}\label{vector proposition}
Let $\chi,\eta \in C^\infty(\Db)$ that vanish on the boundary $\partial \Db$, then the potential vector field $\D\eta$ and solenoidal vector field  $\D^{\perp}\chi$ satisfy the following properties:
\begin{enumerate}
\item $\Lc(\D\eta)(\phi,\psi) = 0$, \quad $\Tc (\D^{\perp}\chi)(\phi,\psi) = 0$.
\item The longitudinal V-line transform of $\D^{\perp}\chi$ is connected with the Radon transform of $\chi$ by the following relation
\begin{align}\label{vector L_dperp_chi}
\Lc (\D^{\perp}\chi)(\phi,\psi) = \frac{\partial}{\partial s}\left[\Rc\chi\left(s,\xi^{-}(\phi,\psi)\right)-\Rc\chi\left(s,\xi^{+}(\phi,\psi)\right)\right], \quad s=\sin{\psi}.
\end{align}
\item The transverse V-line transform of $\D\eta$ is connected with the Radon transform of $\eta$ by the following relation
\begin{align}\label{vector T_d_eta}
\Tc (\D\eta)(\phi,\psi) = -\frac{\partial}{\partial s}\left[\Rc\eta\left(s,\xi^{-}(\phi,\psi)\right) -\Rc\eta\left(s,\xi^{+}(\phi,\psi)\right)\right], \quad s=\sin{\psi}.
\end{align}
\end{enumerate}
\end{prop}
\begin{theorem}[Kernel Description]
Let $\vf \in C^{\infty}_c(S^1(\Db))$. Then
\begin{enumerate}[label=(\alph*)]
\item  $\Lc\vf=0$ if and only if $\vf=\D\eta$, for some smooth function $\eta$ satisfying $\eta|_{\partial \Db}=0$.
\item   $\Tc\vf=0$ if and only if $\vf=\D^\perp \chi$, for some smooth function $\chi$ satisfying $\chi|_{\partial \Db}=0$.
\end{enumerate}
\end{theorem}
\begin{proof}The ``if'' part of the theorem follows from Proposition \ref{vector proposition}, and the ``only if'' part is an implication of inversion formulas derived below. 
% \textbf{Part (a)}. If $\vf=\D\eta,$ then from Proposition \ref{vector proposition}, we have $\Lc\vf=0$. The other direction follows due to the reconstruction of $\chi$ only from $\Lc\vf.$\\
%   \textbf{Part (b)}. If $\vf=\D^\perp\chi,$ then from Proposition \ref{vector proposition}, we have $\Tc\vf=0$. The other direction follows  directly from
%    the reconstruction of $\eta$ only from $\Tc\vf.$  
\end{proof}
\noindent Now, we prove the Theorem \ref{Recovery from M_lf} for $m=1$, that is, we show how to recover a vector field from the knowledge of its longitudinal and transverse V-line transforms ($\Lc\vf,\Tc\vf$).
\begin{proof}[\textbf{Proof of Theorem \ref{Recovery from M_lf} $(m=1)$}]
%     From the above proposition, we have
% \begin{align*}
%     \mathcal{L}\vf(\phi,\psi) & = \Lc(\D^{\perp}\chi)(\phi,\psi)
%      = \frac{\partial}{\partial s}\left[\Rc\chi\left(s,\xi^{-}(\phi,\psi)\right) -\Rc\chi\left(s,\xi^{+}(\phi,\psi)\right)\right].
% \end{align*}
We know that the unknown vector field $\vf$ can be decomposed as follows: 
$$\vf= \D^{\perp}\chi+\D\eta, \quad \chi|_{\partial \Db}=0, \quad \eta|_{\partial \Db}=0.$$
The idea is to use $\Lc \vf$ to recover $\chi$ and use $\Tc \vf $ to recover $\eta$. Keeping this in mind, let us apply $\Lc$ on the above relation to get 
% \begin{align*}
%    \Lc \vf (\phi, \psi) &= \Lc (\D^{\perp}\chi)(\phi,\psi), \quad \mbox{ since } \ \ \Lc (\D\eta) = 0\\
%    &=  \frac{\partial}{\partial s}\left[\Rc\chi\left(s,\xi^{-}(\phi,\psi)\right)-\Rc\chi\left(s,\xi^{+}(\phi,\psi)\right)\right], \quad s=\sin{\psi}.
% \end{align*}
$$ \Lc \vf (\phi, \psi) = \Lc (\D^{\perp}\chi)(\phi,\psi), \quad \mbox{ since } \ \ \Lc (\D\eta) = 0.$$
The right-hand side of the above equation can be computed in terms of the Radon transform of $\chi$ from equation \eqref{vector L_dperp_chi}. Therefore, we integrate equation  \eqref{vector L_dperp_chi} from $s$ to $\infty$ to obtain 
% \begin{align*}
%      F(\phi,\psi) :&= \int_{s}^{\infty}\Lc(\D^{\perp}\chi)(\phi,\sin^{-1}{t}) \,d t\\ & = \int_{s}^{\infty}\frac{\partial}{\partial t}\left[\Rc\chi\left(t,\phi-\sin^{-1}{t}+\frac{\pi}{2}\right) -\Rc\chi\left(t,\phi+\sin^{-1}{t}-\frac{\pi}{2}\right)\right] \,d t\\
%    & =- \Rc\chi\left(s,\xi^{-}(\phi,\psi)\right) +\Rc\chi\left(s,\xi^{+}(\phi,\psi)\right).
% \end{align*}
\begin{align*}
\int_{s}^{\infty}\Lc(\D^{\perp}\chi)(\phi,\sin^{-1}{t}) \,d t & = \int_{s}^{\infty}\frac{\partial}{\partial t}\left[\Rc\chi\left(t,\phi-\sin^{-1}{t}+\frac{\pi}{2}\right) -\Rc\chi\left(t,\phi+\sin^{-1}{t}-\frac{\pi}{2}\right)\right] \,d t\\
   & =- \Rc\chi\left(s,\xi^{-}(\phi,\psi)\right) +\Rc\chi\left(s,\xi^{+}(\phi,\psi)\right)\\
   &= \Vc_w \chi (\phi, \psi), \quad \mbox{ with } c_1 = -c_2 =  -1 \mbox{ (please see equation \eqref{eq:relation between Radon and W V line})}.
\end{align*}
% By using the Fourier series expansion \eqref{fourier component known data}, we have the $n^{th}$ Fourier coefficient of $\Rc\chi$
% \begin{align}\label{vector chi recover from Lf}
%     F_n(\psi) &= -(\Rc\chi)_n(\sin\psi)e^{-\mathrm{i}n\left(\psi-\frac{\pi}{2}\right)} +(\Rc\chi)_n(\sin\psi)e^{\mathrm{i}n\left(\psi-\frac{\pi}{2}\right)} \nonumber \\
%     &= 2\I \sin\left(n\left(\psi-\frac{\pi}{2}\right)\right)(\Rc\chi)_n(\sin\psi)\nonumber\\
%     \Longrightarrow \qquad \qquad (\Rc\chi)_n(\sin\psi) &= \frac{F_n(\psi)}{2\I \sin\left(n\left(\psi-\frac{\pi}{2}\right)\right)}.
% \end{align}
%from one of the formulas \eqref{Recovery formula-I} or \eqref{Recovery formula-II}, we recover $\chi$ explicitly.\\
% Similarly, from the above Proposition, we get
% \begin{align*}
%     \mathcal{T}\vf(\phi,\psi) & = \Tc(\D\eta)(\phi,\psi)
%      = -\frac{\partial}{\partial s}\left[\Rc\eta\left(s,\xi^{-}(\phi,\psi)\right) -\Rc\eta\left(s,\xi^{+}(\phi,\psi)\right)\right].
% \end{align*}
Similarly, by integrating equation \eqref{vector T_d_eta} from $s$ to $\infty$, we obtain
\begin{align*}
\int_{s}^{\infty}\Tc(\D\eta)(\phi,\sin^{-1}{t}) dt &= \Rc\eta\left(s,\xi^{-}(\phi,\psi)\right)-\Rc\eta\left(s,\xi^{+}(\phi,\psi)\right) \\ 
&= \Vc_w \eta (\phi, \psi), \quad \mbox{ with } c_1 = -c_2 =  1.
\end{align*}
% And the corresponding relation between the Fourier coefficients is given by 
% \begin{align}\label{vector eta recover from Tf}
%     G_n(\psi) &= (\Rc\eta)_n(\sin\psi)e^{-\mathrm{i}n\left(\psi-\frac{\pi}{2}\right)}-(\Rc\eta)_n(\sin\psi)e^{\mathrm{i}n\left(\psi-\frac{\pi}{2}\right)}\nonumber\\
%     &= -2\I \sin\left(n\left(\psi-\frac{\pi}{2}\right)\right)(\Rc\eta)_n(\sin\psi)\nonumber \\
%     \Longrightarrow \qquad \qquad 
%     (\Rc\eta)_n(\sin\psi) &= -\frac{G_n(\psi)}{2\I \sin\left(n\left(\psi-\frac{\pi}{2}\right)\right)}.
% \end{align}
From the two relations derived above, we get the weighted V-line transforms (with different weights) of  $\chi$ and $\eta$, respectively. Therefore, we can recover $\chi$ and $\eta$ explicitly by using any of the two inversion formulas of $\Vc_w$ given \eqref{weighted Recovery formula-I} and \eqref{weighted Recovery formula-II}, with appropriate choices of constants $c_1$ and $c_2$.
\end{proof} 
\subsection{Symmetric \texorpdfstring{$m$}{m}-tensor fields}
Again, we start our discussion with a known decomposition result of a symmetric $m$-tensor field discussed in \cite[Theorem 5.1]{derevtsov3}.
\begin{theorem}\label{m-tensor decomp}\cite[Theorem 5.1]{derevtsov3}: For any symmetric $m$-tensor field $\vf\in C^{\infty}_c(S^m(\Db))$, there exist unique smooth functions $\chi^{(j)}, \  0\leq j \leq m$ satisfying:
\begin{align}
   & \vf=  \sum_{j=0}^{m}(\D^{\perp})^{m-j}\D^{j}\chi^{(j)},\nonumber\\
   & \chi^{(j)}|_{\partial \Db}=0,\dots,\D^{(m-1)}\chi^{(j)}|_{\partial \Db}=0, \quad \forall \ 0\leq j \leq m.
\end{align}
\end{theorem}
\noindent With the help of the above decomposition, we have the following proposition that discusses the action of $\Mc_\ell$ on different parts of $\vf$. This proposition is a direct generalization of the \cite[Proposition 5.1]{derevtsov3} in our V-line setting. 
\begin{prop}\label{m-tensor Proposition}
Let $\chi^{(j)}$, $0\leq j \leq m$ be smooth functions in $\Db$ and their derivatives up to order $(m-1)$ vanish on the boundary $\partial\Db$, then the following properties hold:
\begin{enumerate}
\item $\Mc_{\ell}\left((\D^{\perp})^{m-j}\D^{j}\chi^{(j)}\right)(\phi,\psi) = 0,$ $j \neq \ell, \quad \forall \ 0\leq \ell\leq m$, where $\Mc_0 = \Lc$ and $\Mc_m = \Tc$.
%(ii) $\Mc ((\D^{\perp})^2\chi)(\phi,\psi) = 0, \quad \Mc(\D^2\lambda)(\phi,\psi)=0$ \\
%(iii) $\Tc ((\D^{\perp})^2\chi)(\phi,\psi) = 0, \quad \Tc((\D\D^{\perp}\eta)(\phi,\psi)=0$.
\item The mixed V-line transform of the field $(\D^{\perp})^{m-\ell}\D^{\ell}\chi^{(\ell)}$ is connected with the Radon transform of its potential $\chi^{(\ell)}$ by the following relation
% \textbf{Case (1)}: $m=$ odd
\begin{align}
     \label{odd and even m tensor}
 &\Mc_{\ell}\left((\D^{\perp})^{m-\ell}\D^{\ell}\chi^{(\ell)}\right)(\phi,\psi)\nonumber\\ & \qquad = C^{\ell}_m 
	\begin{cases}
   \frac{\partial^m}{\partial s^m}\left[\Rc\chi^{(\ell)}\left(s,\xi^{-}(\phi,\psi)\right) +(-1)^m\Rc\chi^{(\ell)}\left(s,\xi^{+}(\phi,\psi)\right)\right],  \quad  & \ell = \text{even} \\ 
   -\frac{\partial^m}{\partial s^m}\left[\Rc\chi^{(\ell)}\left(s,\xi^{-}(\phi,\psi)\right) +(-1)^m\Rc\chi^{(\ell)}\left(s,\xi^{+}(\phi,\psi)\right)\right],   \quad & \ell = \text{odd} 
	\end{cases}
\end{align}
where $\displaystyle C^{\ell}_{m} = \frac{\ell!(m-\ell)!}{m!}, \quad 0\leq \ell\leq m$ and $ s=\sin{\psi}$.\\
    \end{enumerate}
\end{prop}

\begin{theorem}[Kernel Description]
    Let $\vf \in C^{\infty}_c(S^m(\Db))$. Then $\Mc_{\ell}\vf=0$ if and only if $\displaystyle\vf = \sum_{j \neq \ell}(\D^{\perp})^{m-j}\D^{j}\chi^{(j)}$, $\forall \ 0\leq j,\ell\leq m$ for scalar functions $\chi^{(j)}$ satisfying the boundary conditions $ \chi^{(j)}|_{\partial \Db}=0,\dots,\D^{(m-1)}\chi^{(j)}|_{\partial \Db}=0$.
\end{theorem}
\begin{proof} If $\displaystyle\vf=\sum_{j \neq \ell}(\D^{\perp})^{m-j}\D^{j}\chi^{(j)}$, then the first part of Proposition \ref{m-tensor Proposition} implies $\Mc_{\ell}\vf=0$. The other direction is an implication of inversion formulas derived below. 
\end{proof}
% \noindent Now, we prove the Theorem \ref{Recovery from M_lf} that proves a reconstruction of a symmetric $m$-tensor field $\vf$ from the knowledge of longitudinal, transverse, and mixed V-line transforms ($\Mc_{\ell}\vf, 0\le \ell \le m$). 
\begin{proof}[\textbf{Proof of Theorem \ref{Recovery from M_lf}}]
We know that any symmetric $m$-tensor field $\vf$ can be decomposed as follows:
\begin{align*}
   & \vf=  \sum_{j=0}^{m}(\D^{\perp})^{m-j}\D^{j}\chi^{(j)},\nonumber\\
   & \chi^{(j)}|_{\partial \Db}=0,\dots,\D^{(m-1)}\chi^{(j)}|_{\partial \Db}=0, \quad \forall \ 0\leq j \leq m.
\end{align*}
From the kernel description, we can only hope to recover $\chi^{(\ell)}$ from  $\Mc_{\ell}\vf$. Therefore, applying $\Mc_{\ell}$ on the above relation, we obtain
\begin{align*}
 \Mc_{\ell}\vf(\phi,\psi) =  \Mc_{\ell}\left((\D^{\perp})^{m-\ell}\D^{\ell}\chi^{(\ell)}\right)(\phi,\psi), \quad \text{since}\quad  \Mc_{\ell}\bigg(\sum_{j \neq \ell}(\D^{\perp})^{m-j}\D^{j}\chi^{(j)}\bigg)(\phi,\psi)=0. 
\end{align*}
Now, we repeatedly integrate the above equation \eqref{odd and even m tensor} $m$-times 
%from between the limit $s$ to $\infty$  
to obtain
\begin{align*}
&\underbrace{\int_{s}^{\infty}\int_{t_{m}}^{\infty}\cdots \int_{t_2}^{\infty}}_\textrm{$m$-times}\Mc_{\ell}\vf(\phi,\sin^{-1}t_1)\, dt_1\cdots dt_{m-1}dt_{m} \\& \qquad = C^{\ell}_m 
	\begin{cases}
  (-1)^m \Rc\chi^{(\ell)}\left(s,\xi^{-}(\phi,\psi)\right) +\Rc\chi^{(\ell)}\left(s,\xi^{+}(\phi,\psi)\right),  \quad  & \ell = \text{even} \\ 
   (-1)^{m+1}\Rc\chi^{(\ell)}\left(s,\xi^{-}(\phi,\psi)\right) -\Rc\chi^{(\ell)}\left(s,\xi^{+}(\phi,\psi)\right),   \quad & \ell = \text{odd} 
	\end{cases}
 \\& \qquad =
	\begin{cases}
  \Vc_w\chi^{(\ell)}, \quad \text{with} \quad c_1 = (-1)^m C^{\ell}_m\quad \text{and}\quad c_2 = C^{\ell}_m,   & \ell = \text{even} \\ 
  \Vc_w\chi^{(\ell)},   \quad \text{with} \quad  c_1 = (-1)^{m+1} C^{\ell}_m\quad \text{and}\quad c_2 = -C^{\ell}_m, & \ell = \text{odd}. 
	\end{cases}
\end{align*}
From the relation derived above, we get the weighted V-line transform of $\chi^{(\ell)}$ with appropriate choices of constants $c_1$ and $c_2$. Therefore, we can recover $\chi^{(\ell)}$ explicitly by using any of the two inversion formulas of $\Vc_w$ given in \eqref{weighted Recovery formula-I} and \eqref{weighted Recovery formula-II}.

% Using the Proposition \ref{m-tensor Proposition} and integrating equation \eqref{odd and even m tensor} from $s$ to $\infty$, $m$-times, and then using the Fourier series expansion, we get the $n^{th}$ Fourier coefficient for $\Rc\chi^{(\ell)}$. Hence, from the recovery formula \eqref{Recovery formula-I} or \eqref{Recovery formula-II}, we can recover each $\chi^{(\ell)}$ for $0\leq \ell\leq m$ explicitly.
\end{proof}
%%%%%%%%%%%%%%%%%%%%%%%%%%%%%%%%%%%%%%%%%%%
\section{Proof of Theorem \ref{Recovery from M_lf} (Approach 2)}\label{sec: Compwise recovery}
This section is devoted to studying the componentwise recovery of a symmetric $m$-tensor field from the combination of longitudinal, transverse, and mixed V-line transforms. We present the proof for vector fields and symmetric 2-tensor fields in detail. % We expect the same idea can be used for tensor fields of arbitrary order, but 

% We divide each case in each subsection; more specifically, we will discuss the componentwise recovery of vector fields and symmetric $2$-tensor fields. Then, at the end of the section, we mentioned the remark about the recovery for symmetric $m$-tensor fields. 

\subsection{Vector fields}
% \begin{theorem}
%     Let $\vf$ be a vector field with components in $C_c^{\infty}(\Db)$. Then, $\vf$ can be uniquely recovered from the knowledge of $\Lc\vf$ and $\Tc\vf.$
%     \end{theorem}
In this subsection, we prove Theorem \ref{Recovery from M_lf} for $m=1$, i.e., we derive an inversion formula to recover vector fields from the knowledge of its longitudinal and transverse V-line transforms ($\Lc\vf,\Tc\vf$). 
\begin{proof}[\textbf{Proof of Theorem \ref{Recovery from M_lf} $(m=1)$}]
From Definition \ref{longi V-line} of longitudinal V-line transform of a vector field $\vf$, we have
\begin{align}\label{vector_definition L}
    % \mathcal{L}\vf(\phi,\psi) &= \int_{0}^{\infty} \vuu \cdot \vf(\Phi(\phi) + t\vuu) \,dt +\int_{0}^{\infty} \vv\cdot \vf(\Phi(\phi) + t\vv)\,dt \nonumber\\
    % & =\int_{-\infty}^{\infty} \vuu\cdot \vf\left(\sin (\psi)\Phi\left(\xi^{-}(\phi,\psi)\right) + t\vuu\right) \,dt + \int_{-\infty}^{\infty} \vv\cdot \vf\left(\sin(\psi)\Phi\left(\xi^{+}(\phi,\psi)\right) + t\vv\right)\,dt 
   \Lc\vf(\phi,\psi)&= \Rc(\vuu\cdot \vf)\left(\sin {\psi},\left(\xi^{-}(\phi,\psi)\right)\right) + \Rc( \vv\cdot \vf) \left(\sin{\psi},\left(\xi^{+}(\phi,\psi)\right) \right)
    \nonumber \\
    & =  -\cos(\phi-\psi) \Rc f_{1}\left(\sin{\psi}, \xi^{-}(\phi,\psi)\right) - \sin(\phi-\psi)\Rc f_{2}\left(\sin{\psi}, \xi^{-}(\phi,\psi)\right) \nonumber\\
    & \quad - \cos(\phi+\psi)\Rc f_{1}\left(\sin{\psi}, \xi^{+}(\phi,\psi)\right) - \sin(\phi+\psi)\Rc f_{2}\left(\sin{\psi}, \xi^{+}(\phi,\psi)\right).
\end{align}
Similarly, Definition \ref{trans V-line} of transverse V-line transform of a vector field $\vf$ gives
\begin{align}\label{vector_definition T}
    % \mathcal{T}\vf(\phi,\psi) &=\int_{0}^{\infty} \vuu^{\perp}\cdot \vf(\Phi(\phi) + t\vuu) \,dt  + \int_{0}^{\infty} \vv^{\perp}\cdot \vf(\Phi(\phi) + t\vv)\,dt \nonumber\\
    % & = \int_{-\infty}^{\infty} \vuu^{\perp}\cdot \vf\left(\sin(\psi)\Phi\left(\xi^{-}(\phi,\psi)\right) + t\vuu\right) \,dt + \int_{-\infty}^{\infty} \vv^{\perp}\cdot \vf\left(\sin(\psi)\Phi\left(\xi^{+}(\phi,\psi)\right) + t\vv\right)\,dt \nonumber \\
    \Tc\vf(\phi,\psi)&= \Rc(\vuu^{\perp}\cdot \vf)\left(\sin {\psi},\left(\xi^{-}(\phi,\psi)\right)\right) + \Rc( \vv^{\perp}\cdot \vf) \left(\sin{\psi},\left(\xi^{+}(\phi,\psi)\right) \right)
    \nonumber \\
    & = \sin(\phi-\psi) \mathcal{R}f_{1}\left(\sin{\psi}, \xi^{-}(\phi,\psi)\right) - \cos(\phi-\psi)\Rc f_{2}\left(\sin{\psi}, \xi^{-}(\phi,\psi)\right) \nonumber\\
    & \quad + \sin(\phi+\psi)\Rc f_{1}\left(\sin{\psi}, \xi^{+}(\phi,\psi)\right) -  \cos(\phi+\psi)\Rc f_{2}\left(\sin{\psi}, \xi^{+}(\phi,\psi)\right).
\end{align}
Combining equations \eqref{vector_definition L} and \eqref{vector_definition T} in the following way: 
\begin{align}\label{vector_addition of L_T}
    \Lc \vf(\phi,\psi)+ \mathrm{i} \Tc \vf(\phi,\psi) = &- e^{-\mathrm{i}(\phi-\psi)}\Rc f_{1}\left(\sin{\psi}, \xi^{-}(\phi,\psi)\right) - \mathrm{i}e^{-\mathrm{i}(\phi-\psi)} \Rc f_{2}\left(\sin{\psi}, \xi^{-}(\phi,\psi)\right) \nonumber\\
    & \quad - e^{-\mathrm{i}(\phi+\psi)} \Rc f_{1}\left(\sin{\psi}, \xi^{+}(\phi,\psi)\right) - \mathrm{i} e^{-\mathrm{i}(\phi+\psi)}\Rc f_{2}\left(\sin{\psi}, \xi^{+}(\phi,\psi)\right). \\
\label{vector_substraction of L_T}
    \Lc \vf(\phi,\psi)- \mathrm{i} \Tc \vf(\phi,\psi) = & -e^{\mathrm{i}(\phi-\psi)}\Rc f_{1}\left(\sin{\psi}, \xi^{-}(\phi,\psi)\right) + \mathrm{i}e^{\mathrm{i}(\phi-\psi)} \Rc f_{2}\left(\sin{\psi}, \xi^{-}(\phi,\psi)\right) \nonumber\\
    & \quad -e^{\mathrm{i}(\phi+\psi)} \Rc f_{1}\left(\sin{\psi}, \xi^{+}(\phi,\psi)\right) + \mathrm{i} e^{\mathrm{i}(\phi+\psi)}\Rc f_{2}\left(\sin{\psi}, \xi^{+}(\phi,\psi)\right). 
\end{align}
Let us multiply equations \eqref{vector_addition of L_T} 
 and \eqref{vector_substraction of L_T} by 
$e^{\mathrm{i}(\phi-\psi)}$ and $e^{-\mathrm{i}(\phi-\psi)}$, respectively then adding them gives 
\begin{align}\label{Vector_P}
& P(\phi,\psi):= e^{\mathrm{i}(\phi-\psi)}( \Lc \vf+ \mathrm{i} \Tc \vf)(\phi,\psi) + e^{-\mathrm{i}(\phi-\psi)} (\Lc \vf- \mathrm{i} \Tc \vf)(\phi,\psi) \nonumber\\
& \qquad = - 2\Rc f_1\left(\sin{\psi}, \xi^{-}(\phi,\psi)\right)-2\cos(2\psi)\Rc f_1\left(\sin{\psi}, \xi^{+}(\phi,\psi)\right)\nonumber
\\& \qquad \qquad-2\sin(2\psi)\Rc f_2\left(\sin{\psi}, \xi^{+}(\phi,\psi)\right).
\end{align}
Now, we multiply $e^{\mathrm{i}(\phi+\psi)}$ with \eqref{vector_addition of L_T} and  $e^{-\mathrm{i}(\phi+\psi)}$ with \eqref{vector_substraction of L_T} then by adding them, we get 
\begin{align}\label{Vector_Q}
   & Q(\phi,\psi):= e^{\mathrm{i}(\phi+\psi)}( \Lc \vf+ \mathrm{i} \Tc \vf)(\phi,\psi) + e^{-\mathrm{i}(\phi+\psi)} (\Lc \vf- \mathrm{i} \Tc \vf)(\phi,\psi)\nonumber\\
   & \qquad = -2\cos(2\psi)\Rc f_1\left(\sin{\psi}, \xi^{-}(\phi,\psi)\right) + 2\sin(2\psi)\Rc f_2\left(\sin{\psi}, \xi^{-}(\phi,\psi)\right)\nonumber
   \\& \qquad \qquad-2\Rc f_1\left(\sin{\psi}, \xi^{+}(\phi,\psi)\right). 
\end{align}
Using the Fourier series expansion for equations \eqref{Vector_P} and \eqref{Vector_Q}, we get the $n^{th}$ Fourier coefficients of $P(\phi,\psi)$ and $Q(\phi,\psi)$ as follows
\begin{align}\label{vector_Form of P_n}
    P_n(\psi) &=  - 2(\Rc f_{1})_n(\sin{\psi})e^{-\mathrm{i}n\left(\psi-\frac{\pi}{2}\right)}-2\left\{\cos(2\psi)(\Rc f_{1})_n(\sin{\psi}) + \sin(2\psi) (\Rc f_{2})_n(\sin{\psi})\right\}e^{\mathrm{i}n\left(\psi-\frac{\pi}{2}\right)}.\\
    \label{vector_Form of Q_n}
     Q_n(\psi) &=  -2\left\{\cos(2\psi)(\Rc f_{1})_n(\sin{\psi}) + \sin(2\psi)(\Rc f_{2})_n(\sin{\psi})\right\}e^{-\mathrm{i}n\left(\psi-\frac{\pi}{2}\right)}- 2 (\Rc f_{1})_n(\sin{\psi})e^{\mathrm{i}n\left(\psi-\frac{\pi}{2}\right)}.
\end{align}
Multiplying equation \eqref{vector_Form of P_n} by $e^{-\mathrm{i}n\left(\psi-\frac{\pi}{2}\right)}$  and equation \eqref{vector_Form of Q_n} by $e^{\mathrm{i}n\left(\psi-\frac{\pi}{2}\right)}$ and then adding them, we get
\begin{align*}
   e^{-\mathrm{i}n\left(\psi-\frac{\pi}{2}\right)} P_n(\psi) + e^{\mathrm{i}n\left(\psi-\frac{\pi}{2}\right)} Q_n(\psi)= -4\left[\cos{2\psi} +\cos(2n\left(\psi-\frac{\pi}{2}\right))\right] (\Rc f_{1})_n(\sin{\psi}).
\end{align*}
Hence, we have the $n^{th}$ Fourier coefficient of $\Rc f_{1}$ 
\begin{align}
(\Rc f_{1})_n(\sin{\psi}) &= - \frac{e^{-\mathrm{i}n\left(\psi-\frac{\pi}{2}\right)} P_n(\psi) + e^{\mathrm{i}n\left(\psi-\frac{\pi}{2}\right)} Q_n(\psi)}{ 4\left[\cos{2\psi} +\cos(2n\left(\psi-\frac{\pi}{2}\right))\right] } \\ \label{vector_form of Rf_1}
\implies \quad \quad (\Rc f_{1})_n(s) &= - \frac{e^{-\mathrm{i}n\left(\sin^{-1}s-\frac{\pi}{2}\right)} P_n(\sin^{-1}s) + e^{\mathrm{i}n\left(\sin^{-1}s-\frac{\pi}{2}\right)} Q_n(\sin^{-1}s)}{ 4\left[\cos(2\sin^{-1}s) +\cos(2n\left(\sin^{-1}s-\frac{\pi}{2}\right))\right] },\quad s=\sin{\psi}.
\end{align}
%By using any of two formulas \eqref{Recovery formula-I} and \eqref{Recovery formula-II}, we can recover $f_1$ explicitly.\\
Using the above relation and \eqref{vector_Form of P_n}, we get the $n^{th}$ Fourier coefficient of $\Rc f_{2}$ as follows
 \begin{align}\label{vector_form of Rf_2}
    (\Rc f_{2})_n(\sin\psi) &= \frac{1}{2\sin{2\psi}} \left(-e^{-\mathrm{i}n(\psi-\frac{\pi}{2})}P_n(\psi) - 2(\Rc f_{1})_n(\sin{\psi})\big[e^{-2\mathrm{i}n\left(\psi-\frac{\pi}{2}\right)}+ \cos{2\psi}\big] \right).
 \end{align}
 %Substituting \eqref{vector_form of Rf_1} into \eqref{vector_form of Rf_2}, and then using any of formulas \eqref{Recovery formula-I} and \eqref{Recovery formula-II}, we can recover $f_{2}$ explicitly.
Once, we obtain $n^{th}$ Fourier coefficients of the Radon transforms of $f_1$ and $f_2$ 
in equations \eqref{vector_form of Rf_1} and \eqref{vector_form of Rf_2} then  $f_1$ and $f_2$ are explicitly reconstructed by using any of the two formulas \eqref{Recovery formula-I} and \eqref{Recovery formula-II}.
 \end{proof}
%%%%%%%%%%%%%%%%%%%%%%%%%%%%%%%%%%%%%%%%%%%%%%%%%%%%%%%%%%
\subsection{Symmetric \texorpdfstring{$2$}{2}-tensor fields}
% \begin{theorem}
%     Let $\vf$ be a symmetric $2$-tensor field with components in $C_c^{\infty}(\Db)$. Then $\vf$ can be uniquely recovered from the knowledge of $\Lc\vf,\Tc\vf$ and $\Mc\vf.$
%     \end{theorem}
Here, we use the knowledge of longitudinal, transverse, and mixed V-line transforms ($\Lc\vf,\Tc\vf,\Mc\vf$) to recover a symmetric $2$-tensor field.
\begin{proof}[\textbf{Proof of Theorem \ref{Recovery from M_lf} $(m=2)$}]
From Definition \ref{longi V-line} of longitudinal V-line transform of a symmetric $2$-tensor field, we have
\begin{align}\label{2-tensor_longitudinal}
% \mathcal{L}\vf(\phi,\psi) &= \int_{0}^{\infty} u_{i}u_{j}f_{ij}(\Phi(\phi) + t\vuu) \,dt + \int_{0}^{\infty} v_{i}v_{j}f_{ij}(\Phi(\phi) + t\vv)\,dt \nonumber\\
% & = \int_{-\infty}^{\infty} u_{i}u_{j}f_{ij}\left(\sin (\psi)\Phi\left(\xi^{-}(\phi,\psi)\right) + t\vuu\right) \,dt + \int_{-\infty}^{\infty} v_{i}v_{j}f_{ij}\left(\sin(\psi)\Phi\left(\xi^{+}(\phi,\psi)\right) + t\vv\right)\,dt \nonumber\\
\Lc\vf(\phi,\psi) &=\Rc\left(\left\langle \vuu^2,  \vf  \right\rangle\right)\left(\sin{\psi},\xi^{-}(\phi,\psi) \right)+ \Rc\left(\left\langle \vv^2, \vf \right\rangle\right)\left(\sin {\psi},\xi^{+}(\phi,\psi) \right) \nonumber\\
& = \cos^2(\phi-\psi) \mathcal{R}f_{11}\left(\sin{\psi}, \xi^{-}(\phi,\psi)\right) + 2\cos(\phi-\psi)\sin(\phi-\psi)\Rc f_{12}\left(\sin{\psi}, \xi^{-}(\phi,\psi)\right)\nonumber\\
& \quad + \sin^2(\phi-\psi)\Rc f_{22}\left(\sin{\psi}, \xi^{-}(\phi,\psi)\right) + \cos^2(\phi+\psi)\Rc f_{11}\left(\sin{\psi}, \xi^{+}(\phi,\psi)\right)\nonumber\\
& \quad + 2 \cos(\phi+\psi)\sin(\phi+\psi)\Rc f_{12}\left(\sin{\psi}, \xi^{+}(\phi,\psi)\right) + \sin^2(\phi+\psi)\Rc f_{22}\left(\sin{\psi}, \xi^{+}(\phi,\psi)\right). 
\end{align}
Using Definition \ref{trans V-line} of transverse V-line transform of a symmetric $2$-tensor field, we have
\begin{align}\label{2-tensor_transverse}
% \mathcal{T}\vf(\phi,\psi) &= \int_{0}^{\infty} u^{\perp}_{i}u^{\perp}_{j}f_{ij}(\Phi(\phi) + t\vuu) \,dt + \int_{0}^{\infty} v^{\perp}_{i}v^{\perp}_{j}f_{ij}(\Phi(\phi) + t\vv)\,dt\nonumber\\
% & = \int_{-\infty}^{\infty} u^{\perp}_{i}u^{\perp}_{j}f_{ij}\left(\sin (\psi)\Phi\left(\xi^{-}(\phi,\psi)\right) + t\vuu\right) \,dt + \int_{-\infty}^{\infty} v^{\perp}_{i}v^{\perp}_{j}f_{ij}\left(\sin(\psi)\Phi\left(\xi^{+}(\phi,\psi)\right) + t\vv\right)\,dt\nonumber\\
\Tc\vf(\phi,\psi) &=\Rc\left(\left\langle (\vuu^\perp)^2,  \vf  \right\rangle\right)\left(\sin {\psi},\xi^{-}(\phi,\psi) \right)+ \Rc\left(\left\langle (\vv^\perp)^2, \vf \right\rangle\right)\left(\sin {\psi},\xi^{+}(\phi,\psi) \right) \nonumber\\
& = \sin^2(\phi-\psi) \mathcal{R}f_{11}\left(\sin{\psi}, \xi^{-}(\phi,\psi)\right) - 2\cos(\phi-\psi)\sin(\phi-\psi)\Rc f_{12}\left(\sin{\psi}, \xi^{-}(\phi,\psi)\right)\nonumber\\
& \quad + \cos^2(\phi-\psi)\Rc f_{22}\left(\sin{\psi}, \xi^{-}(\phi,\psi)\right) + \sin^2(\phi+\psi)\Rc f_{11}\left(\sin{\psi}, \xi^{+}(\phi,\psi)\right)\nonumber\\
& \quad - 2 \cos(\phi+\psi)\sin(\phi+\psi)\Rc f_{12}\left(\sin{\psi}, \xi^{+}(\phi,\psi)\right) + \cos^2(\phi+\psi)\Rc f_{22}\left(\sin{\psi}, \xi^{+}(\phi,\psi)\right). 
\end{align}
Similarly, from Definition \ref{mixed V-line} of mixed V-line transform of a symmetric $2$-tensor field, we have
\begin{align}\label{2-tensor_mixed}
% \mathcal{M}\vf(\phi,\psi) &= \int_{0}^{\infty} u^{\perp}_{i}u_{j}f_{ij}(\Phi(\phi) + t\vuu) \,dt + \int_{0}^{\infty} v^{\perp}_{i}v_{j}f_{ij}(\Phi(\phi) + t\vv)\,dt\nonumber \\
% & = \int_{-\infty}^{\infty} u^{\perp}_{i}u_{j}f_{ij}\left(\sin (\psi)\Phi\left(\xi^{-}(\phi,\psi)\right) + t\vuu\right) \,dt + \int_{-\infty}^{\infty} v^{\perp}_{i}v_{j}f_{ij}\left(\sin(\psi)\Phi\left(\xi^{+}(\phi,\psi)\right) + t\vv\right)\,dt\nonumber\\
\Mc\vf(\phi,\psi) &=\Rc\left(\left\langle \vuu^\perp\vuu,  \vf \right\rangle\right)\left(\sin (\psi),\xi^{-}(\phi,\psi) \right) + \Rc\left(\left\langle \vv^\perp\vv,  \vf \right\rangle\right)\left(\sin (\psi),\xi^{+}(\phi,\psi) \right)\nonumber\\
& = -\cos(\phi-\psi) \sin(\phi-\psi) \mathcal{R}f_{11}\left(\sin{\psi}, \xi^{-}(\phi,\psi)\right) + \cos^2(\phi-\psi)\Rc f_{21}\left(\sin{\psi}, \xi^{-}(\phi,\psi)\right)\nonumber\\
& \quad - \sin^2(\phi-\psi)\Rc f_{12}\left(\sin{\psi}, \xi^{-}(\phi,\psi)\right)+\sin(\phi-\psi)\cos(\phi-\psi)\Rc f_{22}\left(\sin{\psi}, \xi^{-}(\phi,\psi)\right) \nonumber\\
& \quad - \sin(\phi+\psi)\cos(\phi+\psi)\Rc f_{11}(\sin{\psi}, \xi^{+}(\phi,\psi)) +  \cos^2(\phi+\psi)\Rc f_{21}\left(\sin{\psi}, \xi^{+}(\phi,\psi)\right)\nonumber\\
& \quad - \sin^2(\phi+\psi)\Rc f_{12}\left(\sin{\psi}, \xi^{+}(\phi,\psi)\right) + \sin(\phi+\psi)\cos(\phi+\psi)\Rc f_{22}\left(\sin{\psi}, \xi^{+}(\phi,\psi)\right). 
\end{align}
Adding equations \eqref{2-tensor_longitudinal} and \eqref{2-tensor_transverse} gives 
\begin{align}\label{addition of L_T}R(\phi,\psi) := \Lc \vf(\phi,\psi)+ \Tc \vf(\phi,\psi) = &\Rc f_{11}\left(\sin{\psi}, \xi^{-}(\phi,\psi)\right) + \Rc f_{22}\left(\sin{\psi}, \xi^{-}(\phi,\psi)\right) \nonumber\\
& \quad + \Rc f_{11}\left(\sin{\psi}, \xi^{+}(\phi,\psi)\right) + \Rc f_{22}\left(\sin{\psi}, \xi^{+}(\phi,\psi)\right). 
\end{align}
%Subtracting \eqref{2-tensor_transverse} from \eqref{2-tensor_longitudinal} to get
%\begin{align*}
%    \Lc \vf(\phi,\psi)- \Tc \vf(\phi,\psi) & = \cos2(\phi-\psi) \mathcal{R}f_{11}\left(\sin(\psi), \xi^{-}(\phi,\psi)\right) + 2 \sin2(\phi-\psi) \mathcal{R}f_{12}\left(\sin(\psi), \xi^{-}(\phi,\psi)\right)\\
%    & \quad - \cos2(\phi-\psi)\mathcal{R}f_{22}\left(\sin(\psi), \xi^{-}(\phi,\psi)\right) + \cos2(\phi+\psi)\Rc f_{11}\left(\sin(\psi), \xi^{+}(\phi,\psi)\right)\\
%    & \quad + 2 \sin2(\phi+\psi)\Rc f_{12}\left(\sin(\psi), \xi^{+}(\phi,\psi)\right) - \cos2(\phi+\psi)\Rc f_{22}\left(\sin(\psi), \xi^{+}(\phi,\psi)\right)
%\end{align*}
Combining equations \eqref{2-tensor_longitudinal}, \eqref{2-tensor_transverse}, and \eqref{2-tensor_mixed} in the following two ways:
\begin{align}\label{difference L_T and sum M}
(\Lc \vf- \Tc \vf + 2\mathrm{i} \Mc \vf)(\phi,\psi) & = e^{-2\mathrm{i}(\phi-\psi)}\Rc f_{11}\left(\sin{\psi}, \xi^{-}(\phi,\psi)\right)+ 2\mathrm{i}e^{-2\mathrm{i}(\phi-\psi)}\Rc f_{12}\left(\sin{\psi}, \xi^{-}(\phi,\psi)\right)\nonumber\\
& \quad - e^{-2\mathrm{i}(\phi-\psi)}\Rc f_{22}\left(\sin{\psi}, \xi^{-}(\phi,\psi)\right) + e^{-2\mathrm{i}(\phi+\psi)}\Rc f_{11}\left(\sin{\psi}, \xi^{+}(\phi,\psi)\right)\nonumber\\
& \quad + 2 \mathrm{i}e^{-2\mathrm{i}(\phi+\psi)}\Rc f_{12}\left(\sin{\psi}, \xi^{+}(\phi,\psi)\right)- e^{-2\mathrm{i}(\phi+\psi)}\Rc f_{22}\left(\sin{\psi}, \xi^{+}(\phi,\psi)\right)\\
% \end{align}
%  Subtracting \eqref{2-tensor_transverse} from \eqref{2-tensor_longitudinal}, multiplying \eqref{2-tensor_mixed} with $2\I$, and then subtracting them, we get
% \begin{align}
\label{difference of L_T_M}
(\Lc \vf- \Tc \vf - 2\mathrm{i} \Mc \vf)(\phi,\psi) & = e^{2\mathrm{i}(\phi-\psi)}\Rc f_{11}\left(\sin{\psi}, \xi^{-}(\phi,\psi)\right)- 2\mathrm{i}e^{2\mathrm{i}(\phi-\psi)}\Rc f_{12}\left(\sin{\psi}, \xi^{-}(\phi,\psi)\right)\nonumber\\
& \quad - e^{2\mathrm{i}(\phi-\psi)}\Rc f_{22}\left(\sin{\psi}, \xi^{-}(\phi,\psi)\right) + e^{2\mathrm{i}(\phi+\psi)}\Rc f_{11}\left(\sin{\psi}, \xi^{+}(\phi,\psi)\right)\nonumber\\
& \quad - 2 \mathrm{i}e^{2\mathrm{i}(\phi+\psi)}\Rc f_{12}\left(\sin{\psi}, \xi^{+}(\phi,\psi)\right)- e^{2\mathrm{i}(\phi+\psi)}\Rc f_{22}\left(\sin{\psi}, \xi^{+}(\phi,\psi)\right).
\end{align}
Multiply equation \eqref{difference L_T and sum M} with $e^{2\mathrm{i}(\phi-\psi)}$ and equation \eqref{difference of L_T_M} with $e^{-2\mathrm{i}(\phi-\psi)}$ then by adding them we get
\begin{align}\label{2-tensor_P}
 P(\phi,\psi) & := \nonumber e^{2\mathrm{i}(\phi-\psi)} \times (\Lc \vf- \Tc \vf + 2\mathrm{i} \Mc \vf)(\phi,\psi) +  e^{-2\mathrm{i}(\phi-\psi)} \times  (\Lc \vf- \Tc \vf - 2\mathrm{i} \Mc \vf)(\phi,\psi)\\
 & \nonumber= 2 (\Rc f_{11}-\Rc f_{22})\left(\sin{\psi}, \xi^{-}(\phi,\psi)\right) + 2\mathrm{i}(e^{-4\mathrm{i}\psi}-e^{4\mathrm{i}\psi})\Rc f_{12}\left(\sin{\psi}, \xi^{+}(\phi,\psi)\right) \\
  & \quad \qquad   + (e^{-4\mathrm{i}\psi}+e^{4\mathrm{i}\psi})(\Rc f_{11}-\Rc f_{22})\left(\sin{\psi}, \xi^{+}(\phi,\psi)\right).
\end{align}
Similarly, we multiply \eqref{difference L_T and sum M} by $e^{2\mathrm{i}(\phi+\psi)}$ and equation \eqref{difference of L_T_M} by $e^{-2\mathrm{i}(\phi+\psi)}$ then adding gives 
\begin{align}\label{2-tensor_Q}
 Q(\phi,\psi) & := \nonumber e^{2\mathrm{i}(\phi+\psi)} \times (\Lc \vf- \Tc \vf + 2\mathrm{i} \Mc \vf)(\phi,\psi) +  e^{-2\mathrm{i}(\phi+\psi)} \times  (\Lc \vf- \Tc \vf - 2\mathrm{i} \Mc \vf)(\phi,\psi)\\
 & \nonumber = (e^{4\mathrm{i}\psi}+e^{-4\mathrm{i}\psi}) (\Rc f_{11}-\Rc f_{22})\left(\sin{\psi}, \xi^{-}(\phi,\psi)\right)+ 2(\Rc f_{11}-\Rc f_{22})\left(\sin{\psi}, \xi^{+}(\phi,\psi)\right)\\
 & \qquad + 2\mathrm{i}(e^{4\mathrm{i}\psi}-e^{-4\mathrm{i}\psi})\Rc f_{12}\left(\sin{\psi}, \xi^{+}(\phi,\psi)\right).
\end{align}
Using the Fourier series expansion for equations \eqref{2-tensor_P} and \eqref{2-tensor_Q}, we get the $n^{th}$ Fourier coefficients of $P(\phi,\psi)$ and $Q(\phi,\psi)$ as follows:
\begin{align}\label{Form of P_n}
P_n(\psi) = & 2((\Rc f_{11})_n-(\Rc f_{22})_n)(\sin{\psi})e^{-\mathrm{i}n\left(\psi-\frac{\pi}{2}\right)} + 4\sin(4\psi) (\Rc f_{12})_n(\sin{\psi})e^{\mathrm{i}n\left(\psi-\frac{\pi}{2}\right)} \nonumber\\
& \quad +2\cos(4\psi)((\Rc f_{11})_n-(\Rc f_{22})_n)(\sin{\psi})e^{\mathrm{i}n\left(\psi-\frac{\pi}{2}\right)}.\\
\label{Form of Q_n}
Q_n(\psi) = & 2\cos(4\psi)((\Rc f_{11})_n-(\Rc f_{22})_n  )(\sin{\psi})e^{-\mathrm{i}n\left(\psi-\frac{\pi}{2}\right)}-4\sin(4\psi)(\Rc f_{12})_n(\sin{\psi})e^{-\mathrm{i}n\left(\psi-\frac{\pi}{2}\right)}\nonumber\\
& \quad 
+ 2 \left((\Rc f_{11})_n - (\Rc f_{22})_n \right)(\sin{\psi})e^{\mathrm{i}n\left(\psi-\frac{\pi}{2}\right)}.
\end{align}
Multiply equation \eqref{Form of P_n} with $e^{-\mathrm{i}n\left(\psi-\frac{\pi}{2}\right)}$ and equation \eqref{Form of Q_n} with $e^{\mathrm{i}n\left(\psi-\frac{\pi}{2}\right)}$ then by adding them, we have 
\begin{align*}
T_n(\psi):&= e^{-\mathrm{i}n\left(\psi-\frac{\pi}{2}\right)} P_n(\psi) + e^{\mathrm{i}n\left(\psi-\frac{\pi}{2}\right)} Q_n(\psi)\\
&= 4 \left(\cos 2n\left(\psi-\frac{\pi}{2}\right)+ \cos4\psi\right)\left((\Rc f_{11})_n
- (\Rc f_{22})_n\right)(\sin{\psi}).
\end{align*}
Hence, we have the $n^{th}$ Fourier coefficient of $(\Rc f_{11}-\Rc f_{22})$ 
\begin{align}\label{ difference of f_(11)_and_f_(22)}
    ((\Rc f_{11})_n- (\Rc f_{22})_n)(\sin{\psi}) = \frac{T_n(\psi)}{4 (\cos 2n(\psi-\frac{\pi}{2})+ \cos4\psi)}.
\end{align}
 Using the Fourier expansion for equation \eqref{addition of L_T}, we get the $n^{th}$ Fourier coefficient of $R(\phi,\psi)$ as below
 \begin{align*}
     R_n(\psi) = 2\cos n\left(\psi-\frac{\pi}{2}\right)((\Rc f_{11})_n + (\Rc f_{22})_n)(\sin{\psi}).
\end{align*}
Therefore, the $n^{th}$ Fourier coefficient of $(\Rc f_{11}+\Rc f_{22})$
\begin{align}\label{addition of f_(11)_and_f_(22)}
   ((\Rc f_{11})_n+ (\Rc f_{22})_n)(\sin{\psi}) = \frac{R_n(\psi)}{2 \cos n(\psi-\frac{\pi}{2})}. 
 \end{align}
From equations \eqref{ difference of f_(11)_and_f_(22)} and \eqref{addition of f_(11)_and_f_(22)}, we get the $n^{th}$ Fourier coefficients of $\Rc f_{11} $ and  $\Rc f_{22}$ as follows
 \begin{align*}
(\Rc f_{11})_n(\sin\psi) = \frac{1}{2}\left[\frac{T_n(\psi)}{4 (\cos 2n(\psi-\frac{\pi}{2})+ \cos4\psi)} + \frac{R_n(\psi)}{2 \cos n(\psi-\frac{\pi}{2})}  \right],  \\ 
(\Rc f_{22})_n(\sin\psi) = \frac{1}{2}\left[\frac{R_n(\psi)}{2 \cos n(\psi-\frac{\pi}{2})} -\frac{T_n(\psi)}{4 (\cos 2n(\psi-\frac{\pi}{2})+ \cos4\psi)}  \right].
 \end{align*}
 We use \eqref{Form of P_n}, to get the following relation to compute the $n^{th}$ Fourier coefficient of $\Rc f_{12}$:
  \begin{align*}
(\Rc f_{12})_n(\sin\psi) &= \frac{1}{4\sin 4\psi} \bigg(e^{-\mathrm{i}n(\psi-\frac{\pi}{2})}P_n(\psi) - 2(\Rc f_{11})_n(\sin{\psi})\left[e^{-2\mathrm{i}n\left(\psi-\frac{\pi}{2}\right)}+ \cos4\psi\right] \\
& +2(\Rc f_{22})_n(\sin{\psi})\left[e^{-2\mathrm{i}n\left(\psi-\frac{\pi}{2}\right)}+\cos4\psi\right] \bigg).
\end{align*}
The above three equations are used to recover the $n^{th}$ Fourier coefficients of $f_{11}$, $f_{22}$ and $f_{12}$ by using any of the two formulas \eqref{Recovery formula-I} and \eqref{Recovery formula-II}. This completes the proof of the Theorem.
\end{proof}
\begin{remark}
We expect a similar approach can be applied for tensor fields of any order, but it will involve large calculations, and one needs to find a compact way to write these in order to present them. At this point, we do not have a way to write a proof for the general $m$. 
% but we expect the same idea is applicable and one can recover a symmetric $m$-tensor field $\vf$, componentwise from the combinations of the longitudinal $(\Lc\vf)$, transverse $(\Tc\vf)$, and $(m-1)$ mixed $(\Mc_\ell\vf)$ V-line transforms. 
\end{remark}
%%%%%%%%%%%%%%%%%%%%%%%%%%%%%%%%%%%%%%%%%%%%%%
\section{Proof of Theorem \ref{Recovery from longi/trans moments}}\label{sec: moments recovery}
This section focuses on the full reconstruction of a symmetric $m$-tensor field from the knowledge of the first $(m+1)$ moment longitudinal/transverse V-line transforms.
More specifically, we discuss the recovery for vector fields and symmetric $m$-tensor fields from their V-line longitudinal moments. The same result is true with a moment of V-line transverse transform as well. To avoid repetition, we present the proof only for the longitudinal case. 
\subsection{Vector fields}
% \begin{theorem}
%     Let $\vf\in C^{\infty}_c(S^1(\Db))$. Then $\vf$ can be recovered from $\Lc\vf$ and $\Lc^1\vf.$
% \end{theorem}
In this subsection, we prove the Theorem \ref{Recovery from longi/trans moments} for a vector field. We show that the knowledge of longitudinal and first moment of longitudinal V-line transforms ($\Lc\vf,\Lc^1\vf$) uniquely recovers a vector field.

\begin{proof}[\textbf{Proof of Theorem \ref{Recovery from longi/trans moments} $(m=1)$}] Recall that a vector field $\vf$ can be decomposed into curl-free and divergence-free parts as follows: 
$$\vf = \D\eta + \D^\perp \chi.$$ 
We know the $\D\eta$ part is in the kernel of $\Lc$, and $\chi$ can be recovered explicitly in terms of $\Lc \vf$. To complete the proof of the theorem, we just need to show that $\eta$ can be recovered with the help of $\Lc^1 \vf$. Consider 
% Using the Theorem \ref{vector decomp} and Proposition \ref{vector proposition}, we know that $\Lc (\D\eta)(\phi,\psi) = 0$. So, from $\Lc\vf$, we can only recover $\chi$ using \eqref{vector L_dperp_chi}. We have considered the data $\Lc\vf$ and $\Lc^1\vf$ to get the full recovery of $\vf$. So, using $\Lc^1\vf$ and the knowledge of reconstructed $\chi$, we have
\begin{align}
& (\Lc^1\vf- \Lc^1(\D^\perp \chi))(\phi,\psi)\nonumber\\ &\quad = \Lc^1(\D \eta)(\phi,\psi) \nonumber\\
&\quad = \int_{-\infty}^{\infty} t \left\langle \vuu, \D\eta\left(\sin (\psi)\Phi\left(\xi^{-}(\phi,\psi)\right) + t\vuu\right)\right\rangle \,dt +\int_{-\infty}^{\infty} t \left\langle \vv, \D\eta\left(\sin(\psi)\Phi\left(\xi^{+}(\phi,\psi)\right)+ t\vv\right) \right\rangle \,dt\nonumber\\
% &\quad= \int_{-\infty}^{\infty}t \vuu\cdot \D \eta\left(\sin (\psi)\Phi\left(\xi^{-}(\phi,\psi)\right) +  t\vuu\right)\,dt  +\int_{-\infty}^{\infty} t \vv\cdot \D\eta\left(\sin(\psi)\Phi\left(\xi^{+}(\phi,\psi)\right) + t\vv\right)\,dt \nonumber \\
&\quad=  \int_{-\infty}^{\infty}t\frac{d \eta}{dt}\left(\sin (\psi)\Phi\left(\xi^{-}(\phi,\psi)\right) + t\vuu\right)\,dt + \int_{-\infty}^{\infty} t \frac{d\eta}{dt}\left(\sin(\psi)\Phi\left(\xi^{+}(\phi,\psi)\right) + t\vv\right)\,dt \nonumber\\
&\quad=  -\int_{-\infty}^{\infty} \eta \left(\sin (\psi)\Phi\left(\xi^{-}(\phi,\psi)\right) + t\vuu\right)\,dt - \int_{-\infty}^{\infty}  \eta\left(\sin(\psi)\Phi\left(\xi^{+}(\phi,\psi)\right) + t\vv\right)\,dt \nonumber \\
&\quad=  -\Rc\eta\left(s,\xi^{-}(\phi,\psi)\right)  -\Rc\eta\left(s,\xi^{+}(\phi,\psi)\right)\nonumber \\
&\quad = - \Vc \eta(\phi, \psi).
\end{align}
Since the left-hand side of the above equation is known and hence we can recover $\eta$ explicitly by using any of the formulas \eqref{weighted Recovery formula-I} and \eqref{weighted Recovery formula-II}.
% Then, using the Fourier expansion formula, we can recover the $n^{th}$ Fourier component of $\eta$. Hence, we can recover $\eta$ from either \eqref{Recovery formula-I} or \eqref{Recovery formula-II}. 
\end{proof}

\subsection{Symmetric \texorpdfstring{$m$}{m}-tensor fields}
% \begin{theorem}
%     Let $\vf\in C^{\infty}_c(S^m(\Db))$. Then $\vf$ can be uniquely reconstructed from $\Lc\vf, \Lc^1\vf,\dots,\Lc^m\vf$.
% \end{theorem}
In this subsection, we prove the Theorem \ref{Recovery from longi/trans moments} for symmetric $m$-tensor fields. We show that the knowledge of first $(m+1)$ moment longitudinal V-line transforms ($\Lc\vf,\Lc^1\vf,\dots,\Lc^m\vf$) uniquely determines a symmetric $m$-tensor field.
\begin{proof}[\textbf{Proof of Theorem \ref{Recovery from longi/trans moments}}] 
Again, let us start by recalling the following decomposition of symmetric $m$-tensor fields:
$$ \vf=  \sum_{j=0}^{m}(\D^{\perp})^{m-j}\D^{j}\chi^{(j)}.$$
We have seen above in the proof of Theorem \ref{Recovery from M_lf} that $\chi^{(0)}$ can be recovered explicitly in term of $\Lc \vf$ and the remaining part $\displaystyle \sum_{j=1}^{m}(\D^{\perp})^{m-j}\D^{j}\chi^{(j)}$ is in the kernel of $\Lc$. 
% Proposition \ref{m-tensor Proposition}, we have
% \begin{align*}
%     \Lc \vf(\phi,\psi) = \Lc \left((\D^{\perp})^{m}\chi^{(0)}\right)(\phi,\psi)
% \end{align*}
% so from $\Lc \vf$, we can recover $\chi^{(0)}$. 
Also, one can notice that $\displaystyle \Lc^{k}\left(\sum_{\ell=k+1}^{m}\left((\D^{\perp})^{m-\ell}\D^{\ell}\chi^{(\ell)}\right)\right)=0,\  1\le k \le m.$ 
Therefore, using $\Lc^1 \vf$ and reconstructed $\chi^{(0)}$, we have
\begin{align*}
 &\left(\Lc^1\vf - \Lc^1\left((\D^\perp)^m\chi^{(0)}\right)\right)(\phi,\psi) = \Lc^1\left((\D^\perp)^{m-1} \D\chi^{(1)}\right)(\phi,\psi) \nonumber \\ 
 &\qquad \qquad =- \frac{\partial^{m-1}}{\partial s^{m-1}}\left[\Rc\chi^{(1)}\left(s,\xi^{-}(\phi,\psi)\right) +(-1)^{m-1}\Rc\chi^{(1)}\left(s,\xi^{+}(\phi,\psi)\right)\right].
\end{align*}
Integrating the above equation  $(m-1)$-times from $s$ to $\infty$, we get a weighted V-line transform of $\chi^{(1)}$ with $c_1 = (-1)^m$ and $c_2 = -1$. Therefore, using any of the formulas \eqref{weighted Recovery formula-I} and \eqref{weighted Recovery formula-II}, we can recover $\chi^{(1)}$ explicitly. \\
% Integrating w.r.t $s$, $(m-1)$-times and then computing the $n^{th}$ Fourier coefficient of above equation, we can recover $\chi^{(1)}$ using \eqref{Recovery formula-I} (or \eqref{Recovery formula-II}).\\
Next, using $\Lc^2\vf$ and reconstructed $\chi^{(0)},\chi^{(1)}$, we have
\begin{align*}
 &\left(\Lc^2\vf -\Lc^2\left((\D^\perp)^{m-1}\D\chi^{(1)}\right)- \Lc^2\left((\D^\perp)^m\chi^{(0)}\right)\right)(\phi,\psi) = \Lc^2\left((\D^\perp)^{m-2} \D^2\chi^{(2)}\right)(\phi,\psi)\nonumber \\
 &\qquad \qquad = 2!\frac{\partial^{m-2}}{\partial s^{m-2}}\bigg[\Rc\chi^{(2)}\left(s,\xi^{-}(\phi,\psi)\right)+(-1)^{m-2}\Rc\chi^{(2)}\left(s,\xi^{+}(\phi,\psi)\right)\bigg].
\end{align*}
Integrating the above equation $(m-2)$-times from $s$ to $\infty$, we get a weighted V-line transform of $\chi^{(2)}$ with $c_1 = 2(-1)^{m}$ and $c_2 = 2$. Therefore, using any of the formulas \eqref{weighted Recovery formula-I} and \eqref{weighted Recovery formula-II}, we can recover $\chi^{(2)}$ explicitly. \\
% Integrating w.r.t $s$, $(m-2)$-times and then computing the $n^{th}$ Fourier coefficient, we can recover $\chi^{(2)}$. \\
Repeating the same process $(k-1)$-more times, we obtain $\chi^{(0)}, \chi^{(1)}, \dots, \chi^{(k-1)}$. For the recovery of $\chi^{(k)}$, we use $\Lc^k\vf$ and known $\chi^{(0)}, \chi^{(1)}, \dots, \chi^{(k-1)}$ to write the following relation:
\begin{align}\label{moment chi_k recovery}
 &\left(\Lc^k\vf - \Lc^k\left((\D^\perp)^{m-(k-1)}\D^{(k-1)}\chi^{(k-1)}\right)\dots - \Lc^{k}\left((\D^\perp)^m\chi^{(0)}\right)\right)(\phi,\psi) = \Lc^k\left((\D^\perp)^{m-k} \D^k\chi^{(k)}\right)(\phi,\psi) \nonumber \\
 &\qquad \qquad = (-1)^{k}k!\frac{\partial^{m-k}}{\partial s^{m-k}}\bigg[\Rc\chi^{(k)}\left(s,\xi^{-}(\phi,\psi)\right) +(-1)^{m-k}\Rc\chi^{(k)}\left(s,\xi^{+}(\phi,\psi)\right)\bigg].
\end{align}
Integrating the above equation $(m-k)$-times from $s$ to $\infty$, we get a weighted V-line transform of $\chi^{(k)}$ with $c_1 = k!(-1)^{m}$ and $c_2 = k!(-1)^k$. Therefore, using any of the formulas \eqref{weighted Recovery formula-I} and \eqref{weighted Recovery formula-II}, we can recover $\chi^{(k)}$, $3\leq k \leq m$ explicitly. This completes the proof of the theorem.
% Integrating w.r.t $s$, $(m-k)$-times and then computing the $n^{th}$ Fourier coefficient of \eqref{moment chi_k recovery},  we can recover $\chi^{(k)}$ for $3\leq k \leq m$ . 
\end{proof}
\section{Acknowledgements}\label{sec: acknowledge}We thank Gaik Ambartsoumian for his suggestions, which helped us improve the article. RM was partially supported by SERB SRG grant No. SRG/2022/000947. AP is supported by UGC, Government of India, with a research fellowship. IZ is supported by the Prime Minister's Research Fellowship from the Government of India.
%%%%%%%%%%%%%%%%%%%%%%%%%%%%%%%%%%%%%%%%%%%%%%%%%%%%%%%%%%%%%%%%%%
\bibliographystyle{plain}
\bibliography{references}

\begin{thebibliography}{10}

\bibitem{Kuchment_homeland}
M.~Allmaras, D.~Darrow, Y.~Hristova, G.~Kanschat, and P.~Kuchment.
\newblock Detecting small low emission radiating sources.
\newblock {\em Inverse Problems and Imaging}, 7(1):47--79, 2013.

\bibitem{Ambartsoumian_2012}
G.~Ambartsoumian.
\newblock Inversion of the {V}-line {R}adon transform in a disc and its applications in imaging.
\newblock {\em Computers \& Mathematics with Applications}, 64(3):260--265, 2012.

\bibitem{amb-book}
G.~Ambartsoumian.
\newblock {\em Generalized {R}adon Transforms and Imaging by Scattered Particles: Broken Rays, Cones, and Stars in Tomography}.
\newblock World Scientific, 2023.

\bibitem{amb-lat_2019}
G.~Ambartsoumian and M.~J. Latifi~Jebelli.
\newblock The {V}-line transform with some generalizations and cone differentiation.
\newblock {\em Inverse Problems}, 35(3):034003, 2019.

\bibitem{Amb_Lat_star}
G.~Ambartsoumian and M.~J. Latifi~Jebelli.
\newblock Inversion and symmetries of the star transform.
\newblock {\em The Journal of Geometric Analysis}, 31(11):11270--11291, 2021.

\bibitem{Gaik_Mohammad_Rohit}
G.~Ambartsoumian, M.~J. Latifi~Jebelli, and R.~K. Mishra.
\newblock Generalized {V}-line transforms in {2D} vector tomography.
\newblock {\em Inverse Problems}, 36(10):104002, 2020.

\bibitem{Gaik_Mohammad_Rohit_numerics}
G.~Ambartsoumian, M.~J. Latifi~Jebelli, and R.~K. Mishra.
\newblock Numerical implementation of generalized {V}-line transforms on {2D} vector fields and their inversions.
\newblock {\em SIAM Journal on Imaging Sciences}, 17(1):595--631, 2024.

\bibitem{Gaik_Rohit_Indrani}
G.~Ambartsoumian, R.~K. Mishra, and I.~Zamindar.
\newblock {V}-line 2-tensor tomography in the plane.
\newblock {\em Inverse Problems}, 40(3):Paper No. 035003, 24, 2024.

\bibitem{Gaik_Rohit_Indrani_numerics}
G.~Ambartsoumian, R.~K. Mishra, and I.~Zamindar.
\newblock V-line tensor tomography: numerical results.
\newblock {\em arXiv preprint arXiv:2405.03249}, 2024.

\bibitem{Ambartsoumian_2013}
G.~Ambartsoumian and S.~Moon.
\newblock {A series formula for inversion of the {V}-line {R}adon transform in a disc}.
\newblock {\em Computers \& Mathematics with Applications}, 66(9):1567--1572, 2013.

\bibitem{Gaik_Souvik_numerics}
G.~Ambartsoumian and S.~Roy.
\newblock Numerical inversion of a broken ray transform arising in single scattering optical tomography.
\newblock {\em IEEE Transactions on Computational Imaging}, 2(2):166--173, 2016.

\bibitem{basko1997analytical}
R.~Basko, G.~L. Zeng, and G.~T. Gullberg.
\newblock Analytical reconstruction formula for one-dimensional {C}ompton camera.
\newblock {\em IEEE Transactions on Nuclear Science}, 44(3):1342--1346, 1997.

\bibitem{basko1998application}
R.~Basko, G.~L. Zeng, and G.~T. Gullberg.
\newblock Application of spherical harmonics to image reconstruction for the {C}ompton camera.
\newblock {\em Physics in Medicine \& Biology}, 43(4):887, 1998.

\bibitem{bhardwaj2024inversion}
R.~Bhardwaj, R.~K. Mishra, and M.~Vashisth.
\newblock Inversion of generalized {V}-line transforms of vector fields in $\mathbb{R}^2 $.
\newblock {\em arXiv preprint arXiv:2404.12479}, 2024.

\bibitem{cormack1963representation}
A.~M. Cormack.
\newblock Representation of a function by its line integrals, with some radiological applications.
\newblock {\em Journal of Applied Physics}, 34(9):2722--2727, 1963.

\bibitem{cree1994towards}
M.~J. Cree and P.~J. Bones.
\newblock Towards direct reconstruction from a gamma camera based on compton scattering.
\newblock {\em IEEE Transactions on Medical Imaging}, 13(2):398--407, 1994.

\bibitem{derevtsov3}
E.~Yu. Derevtsov and I.~E. Svetov.
\newblock Tomography of tensor fields in the plane.
\newblock {\em Eurasian Journal of Mathematical and Computer Applications}, 3(2):24--68, 2015.

\bibitem{FMS-PhysRev-10}
L.~Florescu, V.~A. Markel, and J.~C. Schotland.
\newblock Single-scattering optical tomography: Simultaneous reconstruction of scattering and absorption.
\newblock {\em Physical Review E}, 81:016602, 2010.

\bibitem{FMS-11}
L.~Florescu, V.~A. Markel, and J.~C. Schotland.
\newblock Inversion formulas for the broken-ray {R}adon transform.
\newblock {\em Inverse Problems}, 27(2):025002, 2011.

\bibitem{FMS-09}
L.~Florescu, J.~C. Schotland, and V.~A. Markel.
\newblock Single-scattering optical tomography.
\newblock {\em Physical Review E}, 79:036607, 2009.

\bibitem{gouia2014exact}
R.~Gouia-Zarrad and G.~Ambartsoumian.
\newblock Exact inversion of the {C}onical {R}adon {T}ransform with a fixed opening angle.
\newblock {\em Inverse Problems}, 30(4):045007, 2014.

\bibitem{haltmeier2014exact}
M.~Haltmeier.
\newblock Exact reconstruction formulas for a {R}adon transform over cones.
\newblock {\em Inverse Problems}, 30(3):035001, 2014.

\bibitem{Moon_Haltmeir_Daniela}
M.~Haltmeier, S.~Moon, and D.~Schiefeneder.
\newblock Inversion of the attenuated {V}-line transform with vertices on the circle.
\newblock {\em IEEE Transactions on Computational Imaging}, 3(4):853--863, 2017.

\bibitem{Ilmavirta_tensor}
J.~Ilmavirta and G.~P. Paternain.
\newblock Broken ray tensor tomography with one reflecting obstacle.
\newblock {\em Communications in Analysis and Geometry}, 30(6):1269--1300, 2022.

\bibitem{Ilmavirta_function}
J.~Ilmavirta and M.~Salo.
\newblock Broken ray transform on a {R}iemann surface with a convex obstacle.
\newblock {\em Communications in Analysis and Geometry}, 24(2):379--408, 2016.

\bibitem{jung2015inversion}
C.~Jung and S.~Moon.
\newblock Inversion formulas for cone transforms arising in application of compton cameras.
\newblock {\em Inverse Problems}, 31(1):015006, 2015.

\bibitem{jung2016exact}
C.~Jung and S.~Moon.
\newblock Exact inversion of the cone transform arising in an application of a {C}ompton camera consisting of line detectors.
\newblock {\em SIAM Journal on Imaging Sciences}, 9(2):520--536, 2016.

\bibitem{Kats_Krylov-15}
R.~Krylov and A.~Katsevich.
\newblock Inversion of the broken ray transform in the case of energy-dependent attenuation.
\newblock {\em Physics in Medicine \& Biology}, 60(11):4313--4334, 2015.

\bibitem{Kuchment_Fatma_2017_divegentbeam_crt}
P.~Kuchment and F.~Terzioglu.
\newblock Inversion of weighted divergent beam and cone transforms.
\newblock {\em Inverse Problems and Imaging}, 11(6):1071--1090, 2017.

\bibitem{moon2016determination}
S.~Moon.
\newblock On the determination of a function from its {C}onical {R}adon {T}ransform with a fixed central axis.
\newblock {\em SIAM Journal on Mathematical Analysis}, 48(3):1833--1847, 2016.

\bibitem{Moon_Haltmeir_CRT_2017}
S.~Moon and M.~Haltmeier.
\newblock Analytic inversion of a {C}onical {R}adon {T}ransform arising in application of {C}ompton cameras on the cylinder.
\newblock {\em SIAM Journal on Imaging Sciences}, 10(2):535--557, 2017.

\bibitem{natterer2001mathematics}
F.~Natterer.
\newblock {\em The mathematics of computerized tomography}.
\newblock SIAM, 2001.

\bibitem{nguyen2005radon}
M.~K. Nguyen, T.~T. Truong, and P.~Grangeat.
\newblock Radon transforms on a class of cones with fixed axis direction.
\newblock {\em Journal of Physics A: Mathematical and General}, 38(37):8003, 2005.

\bibitem{perry1975reconstructing}
R.~M. Perry.
\newblock Reconstructing a function by circular harmonic analysis of its line integrals.
\newblock {\em Image Processing for 2-D and 3-D Reconstruction from Projections: Theory and Practice in Medicine and the Physical Sciences}, 1975.

\bibitem{schiefeneder2017radon}
D.~Schiefeneder and M.~Haltmeier.
\newblock The {R}adon transform over cones with vertices on the sphere and orthogonal axes.
\newblock {\em SIAM Journal on Applied Mathematics}, 77(4):1335--1351, 2017.

\bibitem{schonfelder1993_astro}
V.~Schonfelder, H.~Aarts, K.~Bennett, H.~Deboer, J.~Clear, W.~Collmar, A.~Connors, A.~Deerenberg, R.~Diehl, A.~Von~Dordrecht, et~al.
\newblock Instrument description and performance of the imaging gamma-ray telescope {COMPTEL} aboard the {C}ompton {G}amma-{R}ay {O}bservatory.
\newblock {\em Astrophysical Journal}, 1993.

\bibitem{singh1983-partI}
M.~Singh.
\newblock An electronically collimated gamma camera for single photon emission computed tomography. part {I}: Theoretical considerations and design criteria.
\newblock {\em Medical physics}, 10(4):421--427, 1983.

\bibitem{singh1983-partII}
M.~Singh and D.~Doria.
\newblock An electronically collimated gamma camera for single photon emission computed tomography. part {II}: Image reconstruction and preliminary experimental measurements.
\newblock {\em Medical Physics}, 10(4):428--435, 1983.

\bibitem{Fatma_Kuchment_Kunyansky_2018}
F.~Terzioglu, P.~Kuchment, and L.~Kunyansky.
\newblock Compton camera imaging and the cone transform: a brief overview.
\newblock {\em Inverse Problems}, 34(5):054002, 16, 2018.

\bibitem{Truong-v-line}
T.T. Truong and M.K. Nguyen.
\newblock On new {V}-line {R}adon transforms in $\mathbb{R}^2$ and their inversion.
\newblock {\em Journal of Physics A: Mathematical and Theoretical}, 44(7):075206, 2011.

\end{thebibliography}

\end{document}